\documentclass[a4paper]{article}
\usepackage{graphicx}
\usepackage{hyperref}
\usepackage{amssymb,amsmath, amsthm}
\usepackage{float}
\usepackage{biblatex}
\usepackage{enumitem} 
\usepackage{authblk} 
\newcommand{\GL}{\operatorname{GL}_2}
\newcommand{\PG}{\operatorname{PGL}_2}
\newcommand{\F}{\mathbb{F}}
\newcommand{\NR}{\mathcal{NR}}

\theoremstyle{definition}
\newtheorem{definition}{Definition}
\theoremstyle{remark}
\newtheorem{remark}[definition]{Remark}
\newtheorem{example}[definition]{Example}
\theoremstyle{plain}
\newtheorem{lemma}[definition]{Lemma}
\newtheorem{corollary}[definition]{Corollary}
\newtheorem{theorem}[definition]{Theorem}
\newtheorem*{theo}{Theorem}
\newtheorem*{lem}{Lemma}

\newtheorem*{theoR}{Theorem R}
\newtheorem*{maintheo}{Main Theorem}
\newcommand{\K}{\overline{K}\cup\{\infty\}}

\begin{document}
\title{Rational Transformations and Invariant Polynomials}
	
	\author{Max Schulz\\
		{University of Rostock}, Germany\\
		\tt {max.schulz@uni-rostock.de}}

\maketitle
\begin{abstract}
    Rational transformations of polynomials are extensively studied in the context of finite fields, especially for the construction of irreducible polynomials. In this paper, we consider the factorization of rational transformations with (normalized) generators of the field $K(x)^G$ of $G$-invariant rational functions for $G$ a finite subgroup of $\PG(K)$, where $K$ is an arbitrary field. Our main theorem shows that the factorization is related to a well-known group action of $G$ on a subset of monic polynomials. With this, we are able to extend a result by Lucas Reis for $G$-invariant irreducible polynomials. Additionally, some new results about the number of irreducible factors of rational transformations for $Q$ a generator of $\F_q(x)^G$ are given when $G$ is non-cyclic.
\end{abstract}
\section*{Introduction}
Let $K$ be an arbitrary field, $K^{\ast}=K\setminus\{0\}$ the set of its units, $K[x]$ the set of polynomials with coefficients in $K$ and $\mathcal{I}_K$ the set of monic irreducible polynomials in $K[x]$, $K(x)$ the rational function field over $K$ and $\F_q$ the field with $q$ elements. For a rational function $Q(x)\in K(x)$ we always denote its numerator and denominator as $g$ and $h$, i.e. $Q(x)=g(x)/h(x)$. Furthermore, we assume that rational functions are represented as reduced fractions, so $\gcd(g,h)=1$. Recall that the degree of $Q$ is $\deg(Q)=\max\{\deg(g),\deg(h)\}$. The $Q$-transform of a polynomial $F(x)= \sum_{i=0}^ka_ix^i\in K[x]$ is defined as $$F^{Q}(x):=h(x)^{\deg(F)}F\left(\frac{g(x)}{h(x)}\right)=\sum\limits_{i=0}^ka_ig(x)^i h(x)^{k-i}.$$ This is not yet well-defined since for all $a\in K^{\ast}$ we have $$Q(x)=\frac{a\cdot g(x)}{a\cdot h(x)}$$ which leads to $$F^{Q}(x)=\sum\limits_{i=0}^ka_i(a g(x))^i (ah(x))^{k-i}=a^k\cdot \sum\limits_{i=0}^ka_ig(x)^i h(x)^{k-i}.$$ One might make this transformation unambiguous by normalizing either the numerator $g$ of $Q$ or the resulting polynomial $F^Q$. In our setup we most often have that $Q$ satisfies $\deg(g)>\deg(h)$ and if $F,g$ are monic so is $F^Q$.

The transformation $F^{Q}$ is often used for constructing irreducible polynomials of high degree over finite fields starting with an irreducible polynomial and a rational function. There is a rich literature on this topic, for example \cite{abkyu}, \cite{bassaR}, \cite{cohenConst}, \cite{KK}, \cite{meyn} and \cite{reisConstr}. The main criterion in use is \begin{lem}[{\cite[Lemma 1]{cohensLem}}] Let $Q(x)=g(x)/h(x)\in K(x)$ and $F\in K[x]$. Then $F^{Q}$ is irreducible if and only if $F\in K[x]$ is irreducible and $g(x)-\alpha h(x)$ is irreducible over $K(\alpha)[x]$, where $\alpha$ is a root of $F$.
\end{lem} The original version is only stated for finite fields, but the proof does work for arbitrary fields as well. The concrete application of this lemma for arbitrary rational functions and starting polynomials $F$ is very hard, which is why the best one can do is to focus on specific rational functions or ''small'' families of rational functions.

This paper considers two specific $Q$-transformations: The first is $Q$ being a rational function of degree 1, i.e. $$Q(x)=\frac{ax+b}{cx+d}$$ where $ad-bc\neq 0$. The $Q$-transform of $F$ looks like this $$ F^Q(x)=\lambda_{Q,F} (cx+d)^{\deg(F)}F\left(\frac{ax+b}{cx+d}\right),$$ where $\lambda_{Q,F}\in K^{\ast}$ makes the resulting polynomial monic. This transformation preserves the irreducibility and degree of $F$ if $\deg(F)\ge 2$ by the previous lemma. There is another way to interpret this particular $Q$-transformation: Let $\GL(K)$ be the set of invertible $2\times 2$-matrices over $K$ and let \begin{equation}\label{Amat}
    A=\left(\begin{array}{cc}
     a& b\\
     c& d
     \end{array}\right)\in\GL(K). 
\end{equation} We make the convention that if we write $A\in\GL(K)$ then we assume that $A$ is of the form (\ref{Amat}). We consider the \textit{projective general linear group} $\PG(K)=\GL(K)/Z$ over $K$, where $Z=K^{\ast} I_2$ is the set of invertible scalar multiples of the identity matrix $I_2$, which is the center of $\GL(K)$. The group $\PG(K)$ is isomorphic to the set of degree 1 rational functions in $K(x)$, where the multiplication is composition. We denote by $[A]$ the coset of $A$ in $\PG(K)$, that is, $$[A]:=\{\alpha\cdot A|\alpha\in K^{\ast}\}.$$ 
 We define $\ast :\PG(K)\times K[x]\to K[x]$ by \begin{equation}\label{asttrans}
     [A]\ast f(x):=\lambda_{A,f}\cdot (cx+d)^{\deg(f)}f\left(\frac{ax+b}{cx+d}\right),
 \end{equation} where $\lambda_{A,f}\in K^{\ast}$ makes the output-polynomial monic.
We call $f\in K[x]$ \textit{$[A]$-invariant} for an $[A]\in\PG(K)$ if $[A]\ast f(x)=f(x)$. Moreover $f$ is called \textit{$G$-invariant} for a subgroup $G\le \PG(K)$ if it is $[A]$-invariant for all $[A]\in G$.
    It can be shown that an $[A]$-invariant polynomial is also $\langle[A]\rangle$-invariant. There is a substantial amount of literature on this transformation and its variations in the context of finite fields, for example \cite{GarefalakisGL}, \cite{reisphd}, \cite{ReisFQG}, \cite{reisEx}, \cite{Sidel}, \cite{sticht}. For instance, it is shown that this transformation induces a (right) group action of $\PG(K)$ on the set of monic polynomials with no roots in $K$. 
The following theorem shows that $[A]$-invariant irreducible monic polynomials over finite fields are always $Q_{A}$-transformations for specific rational functions $Q_{A}$ depending on $A\in\GL(\F_q)$: 
\begin{theoR}[{\cite[Theorem 6.0.7.]{reisphd}}]
Let $[A]\in\PG(\F_q)$ be an element of order $D=\operatorname{ord}([A])$. Then there exists a rational function $Q_A(x)=g_A(x)/h_A(x)$ of degree $D$ with the property that the $[A]$-invariant monic irreducible polynomials of degree $Dm>2$ are exactly the monic irreducible polynomials of the form \begin{equation*}
    F^{Q_A}(x)=h_A(x)^{m}\cdot F\left(\frac{g_A(x)}{h_A(x)}\right),
\end{equation*} where $\deg(F)=m$. In addition, $Q_A$ can be explicitly computed from $A$.
\end{theoR} This theorem is proved by dividing $\PG(\F_q)$ into four types of conjugacy classes and showing it for a nice representative of each class.

Let $G\le \PG(K)$ be a finite subgroup and for $A\in\GL(K)$ set \begin{equation}\label{mobi}
    [A]\circ x:=\frac{ax+b}{cx+d}.
\end{equation} There exists a rational function $Q_G\in K(x)$ of degree $|G|$ so that $K(x)^G=K(Q_G(x))$ where $$K(x)^G:=\{Q\in K(x)|~ Q([A]\circ x)=Q(x)\text{ for all } [A]\in G\}$$ is the fixed field of $G$ (for reference see \cite{bluher1}). Moreover, every rational function $Q\in K(x)^G$ of degree $|G|$ is a generator of $K(x)^G$, so we can normalize $Q_G$ in such a way that $Q_G(x)=g(x)/h(x)$ with $0\le \deg(h)<\deg(g)=|G|$ and $g$ monic. Based on \cite{bluher1}, we call these generators \textit{quotient maps} for $G$ and this is the second class of rational functions we consider in this paper. In \cite{reisphd} it is noted that for some $[A]\in\PG(\F_q)$ the functions $Q_A$ in Theorem R are in fact generators of the fixed field $K(x)^{\langle [A]\rangle}$.

A natural question to ask is whether the function $Q_A$ in Theorem R is always a generator of $K(x)^{\langle [A]\rangle}$ for all $[A]\in\PG(\F_q)$. An understanding of this question is of interest since many constructions of irreducible polynomials over finite fields via $Q$-transformations that work very well use specific generators of specific fields of invariant functions, see for example \cite{cohenConst}, \cite{meyn} and \cite{reisConstr}. Another natural question is whether the theorem still holds if we consider $G$-invariant and not necessarily irreducible polynomials for arbitrary finite subgroups of $\PG(K)$. These two questions led us to study the $Q_G$-transformations of irreducible polynomials and their factorization. We did not want to necessarily restrict ourselves to the case that $K$ is finite, so we formulate the results for arbitrary fields. However, the theory is especially beautiful in characteristic $p>0$ because the finite subgroups of $\PG(K)$ are more diverse there (see \cite{pgl}, \cite{subpglk} and \cite{madan}).

The main result and starting point of this paper can be summarized as the following theorem about the factorization of $F^{Q_G}$ for $F$ an irreducible monic polynomial and $Q_G$ a quotient map for $G$: \begin{maintheo}
    Let $F\in K[x]$ be monic and irreducible, $G\le\PG(K)$ a finite subgroup and $Q_G=g/h\in K(x)$ a quotient map for $G$. Then there is an irreducible monic polynomial $r\in K[x]$ with $\deg(F)|\deg(r)$ and an integer $k>0$ such that $$F^{Q_G}(x)=\left(\prod\limits_{t\in G\ast r}t(x)\right)^k,$$ where $G\ast r:=\{[A]\ast r|[A]\in G\}$ is the $G$-orbit of $r$. Additionally $k=1$ for all but finitely many irreducible monic polynomials $F\in K[x]$.
\end{maintheo} The main difficulty of the proof is to show that $k=1$ for all but finitely many irreducible and monic $F\in K[x]$ in non-perfect fields.

We want to point out that a very similar result is known for the case that $F\in\mathcal{I}_K$ is of degree 1 and $K=\F_q$; we state said theorem for convenience: \begin{theo}[{\cite[Theorem 26]{guire}}]
Let $G$ be a subgroup of $\PG(\F_q)$ and $Q(x)=g(x)/h(x)$ a generator for $\F_q(x)^G$. Let $\alpha\in\overline{\F}_q$ have the property that $Q(\alpha)\in\F_q$ and assume that $G$ acts regularly on the roots of $F_{\alpha}(T):=g(T)-Q(\alpha)h(T)\in\F_q[T]$ via M\"obius-Transformation, then \begin{enumerate}
    \item $F_{\alpha}$ will factor into irreducible polynomials of the same degree over $\F_q[T]$
    \item The minimal polynomial of $\alpha$ is one of the factors of $F_{\alpha}$
    \item The degree of each factor must be the order of an element of $G$.
\end{enumerate}
\end{theo} To see that both theorems are connected notice that for $\beta=Q(\alpha)\in\F_q$ we have that $F^{Q_G}(T)=g(T)-\beta h(T)$ for $F=T-\beta$, which factors into a $G$-orbit of an irreducible polynomial by our Main Theorem and all elements in a $G$-orbit have the same degree, which explains item 1. The second item is also true in our setup, that is, if $\beta\in\overline{K}$ is a root of $F$, then $\alpha\in Q_G^{-1}(\beta)$ is a root of $F^{Q_G}$. The third item, however, is a finite field specific result and generalizes to arbitrary fields and irreducible polynomials $F$ of arbitrary degree as follows: Every irreducible factor of $F^{Q_G}$ has degree $\deg(F)$ times the size of a subgroup of $G$. The condition that the set of roots of $F_{\alpha}$ only contains regular $G$-orbits is a crucial one for the case that $k=1$ in the Main Theorem. All of this will be explained in depth in this paper. The phenomenon that $F^{Q_G}$ factorizes into a $G$-orbit of an irreducible polynomial was, until now, only noted for some instances of generators of specific invariant rational function fields over finite fields.
\begin{example}\begin{enumerate}
    \item We start with $Q_1=x+1/x\in\F_q(x)$ and look at the factorization of $F^{Q_1}$, where $F\in\mathcal{I}_q:=\mathcal{I}_{\F_q}$. It is proved in \cite[Lemma 4]{meyn} that $F^{Q_1}$ is either irreducible and \textit{self-reciprocal} or factorizes into a \textit{reciprocal pair}. For $r\in\mathcal{I}_q\setminus\{x\}$ we set $r^{\ast}(x):=a_0^{-1}x^{\deg(r)}r(1/x)$ as its reciprocal polynomial, where $a_0$ is the constant term of $r$. A polynomial is said to be self-reciprocal if $r(x)=r^{\ast}(x)$ and a reciprocal pair is a pair $r,r^{\ast}$ such that $r\neq r^{\ast}$. This result can be explained with our Main Theorem: Let $$G_1=\left\langle\left[\left(\begin{array}{cc}
     0& 1\\
     1& 0
     \end{array}\right)\right]\right\rangle.$$ This is a subgroup of order 2 and a generator of $G_1$ is $Q_{G_1}(x)=x+1/x=(x^2+1)/x\in K(x)$. Then, for all but finitely many irreducible monic polynomials $F\in \mathcal{I}_K$ we obtain that there exists an irreducible monic polynomial $r\in K[x]$ such that $$F^{Q_{G_1}}(x)=\begin{cases} r(x),&\text{ if }F^{Q_{G_1}} \text{ is irreducible}\\
     r(x)\cdot a_0^{-1}x^{\deg(r)}r(1/x),& \text{ if }F^{Q_{G_1}} \text{ is not irreducible}
     \end{cases}.$$
     \item The factorization of $F(x^n)$ in $\F_q[x]$ for $n|q-1$ leads to another nice example. Let $a\in\F_q^{\ast}$ be a primitive $n$-th root of unity. It can be shown that for all $F\in\mathcal{I}_q\setminus\{x\}$ there exists $m\mid n$ and $r\in\mathcal{I}_q$ such that $$F(x^n)=\prod\limits_{i=0}^{m-1}a^{-i\cdot \deg(r)}\cdot r(a^ix).$$ For reference see \cite{albert} and \cite{Daykin}; for a nice application of this see \cite{Maurin}. The rational function $Q_2(x)=x^n$ is a quotient map for the subgroup $$G_2:=\left\{\left[\left(\begin{array}{cc}
     a^i& 0\\
     0& 1
     \end{array}\right)\right]|i\in\mathbb{N}\right\}$$ and $F^{Q_2}(x)=F(x^n)$. The factors belong to the same $G_2$-orbit, since $$\left[\left(\begin{array}{cc}
     a^i& 0\\
     0& 1
     \end{array}\right)\right]\ast r(x)=a^{-i\deg(r)}\cdot r(a^i x).$$
     \item In \cite{comppoly} it is noted that if $F(x^p-x)$ is not irreducible for $F\in\mathcal{I}_q$ and $p^l=q$, then $F(x^p-x)$ factorizes into exactly $p$ irreducible polynomials of degree $\deg(F)$ and, more precisely, there exists $r\in\mathcal{I}_q$ with $\deg(r)=\deg(F)$ such that $$F(x^p-x)=r(x)\cdot r(x+1)\cdot\ldots \cdot r(x+(p-1)).$$ The rational function $Q_3(x)=x^p-x$ is a quotient map for $$G_3:=\left\{\left[\left(\begin{array}{cc}
     1& a\\
     0& 1
     \end{array}\right)\right]|a\in \F_p\right\}$$ and for $r\in K[x]$ the transformation with an element of $G_3$ looks like this $$\left[\left(\begin{array}{cc}
     1& a\\
     0& 1
     \end{array}\right)\right]\ast r(x)=r(x+a).$$
\end{enumerate}
\end{example}
All of these examples still hold in every field $K$ in which the corresponding subgroups of $\PG(K)$ exist. The factorization of $F^{Q_G}$ can be easily obtained by finding just one irreducible factor and calculating the $G$-orbit of this factor. We also see that in literature, apart from \cite[Theorem 26]{guire}, only small cyclic subgroups of $\PG(\F_q)$ of prime order were considered. In contrast, we want to look at big subgroups of $\PG(\F_q)$ instead. We can obtain the following new result over finite fields:
\begin{theorem}\label{facfin}
Let $K=\F_{q}$ and $G\le \PG(\F_q)$ with quotient map $Q_G\in \F_{q}(x)$. Moreover, set $\mu_G\in \mathbb{N}$ as the maximal order of an element in $G$ and let $F\in \F_q[x]$ be an irreducible monic polynomial such that $F^{Q_G}$ is separable. Then we have that $F^{Q_G}$ has at least $|G|/\mu_G$ irreducible factors and every such factor has degree at most $\mu_G\cdot\deg(F)$.
\end{theorem}
\begin{remark} The polynomial $F^{Q_G}$ is separable if $F\in\mathcal{I}_q$ and $\deg(F)\ge 3$, so the only exception polynomials for which the theorem does not necessarily hold are irreducible polynomials of degree less than 3. For an explanation see Theorem \ref{main1}, Theorem \ref{k=1} and Lemma \ref{deg2}.
\end{remark}
For example, take $\{0\}\neq V\le_p \F_q$ as a $\F_p$-subspace of $\F_q$, then define $$\overset{\sim}{V}:=\left\{\left[\left(\begin{array}{cc}
     1& v\\
     0& 1
     \end{array}\right)\right]|v\in V \right\}.$$ We call $\overset{\sim}{V}$ the to $V$ associated subgroup in $\PG(\F_q)$. Observe that $V\cong \overset{\sim}{V}$ as groups. A quotient map for $\overset{\sim}{V}$ is the to $V$ associated subspace polynomial, that is, $$Q_V(x)=\prod\limits_{v\in V}(x-v)\in\F_q[x].$$ Every non-trivial element in $\overset{\sim}{V}$ has order $p$, so $\mu_{\overset{\sim}{V}}=p$ and therefore we obtain the following corollary \begin{corollary}\label{facunipot}
Let $\{0\}\neq V\le_p \F_q$ be an $\F_p$-subspace of $\F_q$ and $Q_V\in\F_q[x]$ the associated subspace polynomial. For every irreducible (monic) polynomial $F\in K[x]$ we have that $F(Q_{V}(x))$ has at least $|V|/p$ irreducible factors and every irreducible factor has the same degree, which is at most $p\cdot \deg(F)$.
     \end{corollary}
In the last part of this paper we consider two further examples of big subgroups of $\PG(\F_q)$ and show how to apply Theorem \ref{facfin} to them. 

The Main Theorem shows that the irreducible factors of $F^{Q_G}$ belong to the same $G$-orbit. Together with the fact that for every $G$-orbit $G\ast r$ in $\mathcal{I}_K$ there exists an irreducible $F\in\mathcal{I}_K$ such that $F^{Q_G}$ has all polynomials in $G\ast r$ as its factors we can prove a generalization of Theorem R: \begin{theorem}\label{Reisgen}
All but finitely many $G$-invariant irreducible monic polynomials $f$ can be written as a $Q_G$-transformation, i.e. there is $F\in\mathcal{I}_K$ such that $f=F^{Q_G}$.
\end{theorem} This result does not say anything about the existence of $G$-invariant irreducible polynomials, it just makes a statement about them if they exist in $K[x]$!

Our proof of a general version of Theorem R avoids the original idea of dividíng $\PG(K)$ into different types of conjugacy classes and showing the theorem for each type, which should be hard as such a list can become quite large depending on the field, see \cite{subpglk} or \cite{madan}. 

The last result shows that the $G$-invariant but not-necessarily irreducible polynomials are a product of a $Q_G$-transformation and some exception polynomials: \begin{theorem}\label{arbginv}
    Let $G\le\PG(K)$ be a finite subgroup and $Q_G$ a quotient map. There exists $k\in \mathbb{N}\setminus\{0\}$ and irreducible monic polynomials $r_1,\ldots r_k\in K[x]$ and $n_1,\ldots n_k\in \mathbb{N}\setminus\{0\}$ such that for every $G$-invariant monic polynomial $f\in K[x]$ there is a unique monic $F\in K[x]$ and integers $k_i< n_i$ such that \begin{equation*}
     f=\left(\prod\limits_{i=1}^k(\prod\limits_{t\in G\ast r_i} t)^{k_i}\right)\cdot F^{Q_G}.
\end{equation*} 
\end{theorem}  We give full explanations about what the polynomials $r_1,\ldots,r_k$ and the integers $n_i$ are in section 3.

\section{Preliminaries}
\subsection{Invariant Polynomials}
We denote by $\circ:\PG(K)\times (\K)\to \K$ the M\"obius-Transformation on $\K$, that is, $$[A]\circ v=\frac{av+b}{cv+d}.$$ This equation is self-explanatory if $v\notin\{\infty,-\frac{d}{c}\}$. For $c\neq 0$ we set $[A]\circ \infty=\frac{a}{c}$ and $[A]\circ(-\frac{d}{c})=\infty$; $[A]\circ \infty=\infty$ if $c=0$. The M\"obius-Transformation is a left group action of $\PG(K)$ on $\K$ and thus every subgroup of $\PG(K)$ acts on $\K$ too. For $G\le \PG(K)$ we denote the $G$-orbit of $v\in\K$ as $G\circ v$. Let $G\le \PG(K)$ then define $$\NR_K^G:=\{f\in K[x]|f \text{ monic and } f(\alpha)\neq 0\text{ for all }\alpha\in G\circ\infty\}.$$ This set is closed under multiplication, i.e. is a submonoid of $K[x]$. We make the convention that $f(\infty)=\infty$ for all polynomials of degree greater than 0 and $a(\infty)=a$ for $a\in K$. The following basic result about $\ast$ holds: \begin{lemma}
\label{basic}
Let $G\le\PG(K)$. For all $f,g\in\mathcal{NR}_K^G$ and $[A],[B]\in G$ the following hold: \begin{enumerate}
    \item $\deg([A]\ast f)=\deg(f)$
    \item $[AB]\ast f=[B]\ast([A]\ast f)$ and $[I_2]\ast f=f$, so $\ast$ is a right group action of $G$ on $\mathcal{NR}_K^G$
    \item $[A]\ast (fg)=([A]\ast f)([A]\ast g)$
    \item $f$ irreducible if and only if $[A]\ast f$ irreducible
\end{enumerate}
\end{lemma}
We omit the proof as it can be done almost exactly as in \cite{GarefalakisGL} or \cite{sticht}. Because of the fourth item of the previous lemma we know that $G$ induces a group action on $$\mathcal{I}_K^G:=\mathcal{I}_K\cap \NR_K^G.$$ Remember that we write $G\ast f$ for the $G$-orbit of $f$. Note that $G\ast f\subset \NR_K^G$ if $f\in\NR_K^G$ and every polynomial in the orbit has the same degree as $f$. The following lemma explains the connection between $G$-invariant polynomials and the M\"obius-Transformation and the proof can be done similarly as in \cite{GarefalakisGL} or \cite{sticht} again: \begin{lemma}\label{stich}
    Let $G\le \PG(K)$ and $f\in\mathcal{NR}_K^G$. Further we denote by \begin{equation}\label{roots}
    R_f:=\{v\in\overline{K}|f(v)=0\}
\end{equation} the set of roots of $f$ in $\overline{K}$. Then the following hold:\begin{enumerate}
    \item If $f$ is $G$-invariant, then $[A]\circ R_f=R_f$ for all $[A]\in G$. Here $[A]\circ R_f:=\{[A]\circ v|v\in R_f\}$
    \item If $f$ is irreducible the converse is also true, more precisely: $[A]\circ R_f=R_f$ for all $[A]\in G$ implies that $f$ is $G$-invariant
\end{enumerate}
\end{lemma} From now on we use $R_f$ as the set of roots of a polynomial $f\in K[x]$ as defined in (\ref{roots}).
In the lemma above we did not make the assumption that $G$ has to be a finite subgroup of $\PG(K)$. So now, we want to explore what happens if $G$ is infinite. Let $[A]\in\PG(K)$, then it is quite obvious that all fixed points of $[A]$ in $\overline{K}$ under $\circ$ are in $K\cup\{\infty\}$ or in a quadratic extension of $K$. Let $v\in\overline{K}$ with $[K(v):K]\ge 3$, then $[A]\circ v\neq v$ for all $[A]\in \PG(K)$, so $G\circ v$ contains infinitely many elements. Therefore there can not exist $G$-invariant irreducible monic polynomials of degree greater than 2 for $G$ an infinite subgroup of $\PG(K)$, since otherwise it would have infinitely many roots by Lemma \ref{stich}. This is one of the reasons why we focus on finite subgroups of $\PG(K)$. Thus, from now on, $G$ denotes a finite subgroup of $\PG(K)$.
The following corollary helps us to understand the factorization of $G$-invariant polynomials: \begin{corollary}\label{monoid}
Let $f,s,t\in\mathcal{NR}_K^G$, where $f$ is a $G$-invariant polynomial with irreducible factor $r\in\mathcal{I}_K^G$. Then the following hold: \begin{enumerate}
    \item $[A]\ast r$ divides $f$ for all $[A]\in G$
    \item If $\gcd(s,t)=1$, then $\gcd([A]\ast s,[A]\ast t)=1$ for all $[A]\in G$
    \item If $r^n|f$ and $r^{n+1}\nmid f$, then $([A]\ast r)^n|f$ and $([A]\ast r)^{n+1}\nmid f$ for all $[A]\in G$. So all polynomials in $G\ast r$ divide $f$ with the same multiplicity
\end{enumerate}
\end{corollary}
\begin{proof}
The first statement is immediate, since $[A^{-1}]\circ R_r=R_{[A]\ast r}\subset R_f$ by Lemma \ref{stich}, so $[A]\ast r$ divides $f$ as well. 

The condition $\gcd(s,t)=1$ is equivalent to $R_s\cap R_t=\varnothing$ in $\overline{K}$ and again $[A^{-1}]\circ R_s=R_{[A]\ast s}$ and $[A^{-1}]\circ R_t=R_{[A]\ast t}$. With the fact that every $[A]\in G$ induces a bijection on $\overline{K}\cup\{\infty\}$ we obtain $R_{[A]\ast s}\cap R_{[A]\ast t}=\varnothing$.

For the last item we use the previous statement: Write $f=r^n\cdot P$ where $\gcd(r,P)=1$. Then, with the third item of Lemma \ref{basic}
\begin{align*}
    f=[A]\ast f=[A]\ast(r^n\cdot P)=([A]\ast r)^n\cdot ([A]\ast P)
\end{align*} and $\gcd([A]\ast r,[A]\ast P)=1$. 
\end{proof}
We just showed that every $G$-invariant polynomial $f\in\NR_K^G$ consists of powers of $G$-orbits in $\mathcal{I}_K^G$ that are glued together by multiplication. So, in a nutshell, $G$-orbits in $\mathcal{I}_K^G$ are the atoms of $G$-invariant polynomials and thus are quite important for this paper; hence the following definition: \begin{definition}
    We call $f\in\mathcal{NR}_K^G$ $G$-orbit polynomial (or simply orbit polynomial) if there exists an irreducible polynomial $r\in\mathcal{I}_K^G$ such that \begin{equation*}
        f=\prod\limits_{t\in G\ast r} t=:\prod(G\ast r).
    \end{equation*}
\end{definition}
For the sake of completeness we state our observation about the factorization of $G$-invariant polynomials as a corollary, but we omit the proof since it is a trivial consequence of Corollary \ref{monoid}.
\begin{corollary}\label{fac1}
Let $f\in\mathcal{NR}_K^G$ be a $G$-invariant polynomial. Then there are $r_1,\ldots,r_k\in\mathcal{I}_K^G$ and $n_1,\ldots,n_k\in\mathbb{N}\setminus\{0\}$ such that \begin{equation*}
    f=\prod\limits_{i=1}^k(\prod(G\ast g_i))^{n_i}.
\end{equation*}
\end{corollary}
\subsection{Quotient Maps and Rational Transformations}
Throughout the rest of the paper we assume that $Q_G\in K(x)$ is a quotient map for $G$ with monic numerator polynomial $g$. Note that such a rational function exists for all finite subgroups of $\PG(K)$ and if $Q_G'$ is another quotient map for $G$, then there are constants $a,b\in K$ such that $Q_G'(x)=aQ_G(x)+b$ (see \cite{bluher1}). We denote by $(\K)/G$ the set of $G$-orbits in $\K$. The following theorem is an important tool for proving our Main Theorem and explains the name "quotient map": \begin{theorem}[{\cite[Proposition 3.9]{bluher1}}]\label{bij}
The quotient map $Q_G$ induces a bijection between $(\overline{K}\cup\{\infty\})/G$ and $\overline{K}\cup\{\infty\}$, more precisely \begin{align*}
    \psi:\begin{cases} (\overline{K}\cup\{\infty\})/G\to \overline{K}\cup\{\infty\},\\
       G\circ v\mapsto Q_G(G\circ v)=Q_G(v)\end{cases}
\end{align*} is a bijection and $\psi(G\circ\infty)=Q_G(\infty)=\infty$.
\end{theorem}
There are essentially two ways to calculate a quotient map for a given finite subgroup $G$ that we know of. One of them is explained in \cite{guire} and works as follows: Calculate the polynomial $$F_G(y):=\prod\limits_{[A]\in G}(y-([A]\circ x))\in K(x)[y].$$ One of the coefficients of $F_G(y)$ has to be a generator of $K(x)^G$ and thus can be normalized so that it becomes a quotient map. Another method is explained in subsection 3.3. in \cite{bluher1}.

For an arbitrary rational function $Q(x)=g(x)/h(x)$ with $g,h$ having leading coefficients $a(g),a(h)\in K^{\ast}$ we set $$Q(\infty):=\begin{cases}
    \infty, &\deg(g)>\deg(h)\\
    \frac{a(g)}{a(h)}, &\deg(g)=\deg(h)\\
    0, &\deg(h)>\deg(g).
\end{cases}$$ We collect some known facts about rational transformations in the next lemma: \begin{lemma}\label{basicqg}
Let $Q=\frac{g}{h}$ be such that $g$ is monic and $F\in K[x]$ such that $F(Q(\infty))\neq 0$, then the following hold: \begin{enumerate}
    \item $\deg(F^{Q})=\deg(F)\cdot \deg(Q)$
    \item If $F$ is reducible so is $F^{Q}$, more precisely: $F=rt$ and $\deg(r),\deg(t)\ge 1$, then $F^{Q}=r^{Q}t^{Q}$
    \item If $F$ is monic and $\deg(g)>\deg(h)$, then $F^{Q}$ is monic as well
\end{enumerate}
\end{lemma}
The following lemma shows that rational transformations with $Q_G$ yield $G$-invariant polynomials: \begin{lemma}\label{qginv}
Let $F\in K[x]$ be a monic polynomial, then $F^{Q_G}\in\mathcal{NR}_K^G$ and $F^{Q_G}$ is $G$-invariant.
\end{lemma}
\begin{proof}
    By Theorem 3.10 and Proposition 3.4 in \cite{bluher1}, the denominator of $Q_G$ is of the form \begin{equation*}
    h(x)=\prod\limits_{v\in (G\circ \infty)\setminus\{\infty\}}(x-v)^{m_{\infty}},
\end{equation*} where $m_\infty:=|\operatorname{Stab}_G(\infty)|=\frac{|G|}{|G\circ\infty|}$ is the cardinality of the stabilizer of $\infty$ in $G$. First of all we want to show that $F^{Q_G}\in\mathcal{NR}_K^G$. For that we write $F(x)=\sum\limits_{i=0}^k a_ix^i$ with $a_k=1$, then \begin{equation*}
    F^{Q_G}(x)=\sum\limits_{i=0}^k a_ig(x)^i h(x)^{k-i}.
\end{equation*} Since the set of roots of $h$ is $G\circ \infty\setminus\{\infty\}$ we obtain for all $w\in G\circ \infty\setminus\{\infty\}$: \begin{align*}
    F^{Q_G}(w)=\sum\limits_{i=0}^k a_ig(w)^i h(w)^{k-i}=g(w)^k\neq 0.
\end{align*} The last step in the calculation is a consequence of $\gcd(h,g)=1$ which implies $g(w)\neq 0$. Thus we got $F^{Q_G}\in\mathcal{NR}_K^G$, since $\mathcal{NR}_K^G$ is the set of monic polynomials with no roots in the orbit of $\infty$ and $f(\infty)=\infty\neq 0$ for all polynomials\footnote{Note that the only polynomial of degree 0 in $\mathcal{NR}_K^G$ is $1$} with $\deg(f)\ge 1$. 

For the calculation that $F^{Q_G}$ is indeed $G$-invariant we verify whether for all $[A]\in G$ there exists $\alpha_{A}\in K^{\ast}$ such that \begin{equation*}
   (cx+d)^{\deg(F^{Q_G})}F^{Q_G}(\frac{ax+b}{cx+d})=\alpha_{A} F^{Q_G}(x)
\end{equation*} Writing the left side out gives \begin{align*}
    (cx+d)^{\deg(F^{Q_G})}F^{Q_G}((A\circ x))=(cx+d)^{\deg(F^{Q_G})}h(\frac{ax+b}{cx+d})^{\deg(F)}F(Q_G(\frac{ax+b}{cx+d})).
\end{align*} Since $Q_G(\frac{ax+b}{cx+d})=Q_G(x)$, we can focus on $(cx+d)^{\deg(F^{Q_G})}h(\frac{ax+b}{cx+d})^{\deg(F)}$.
For that we have to consider 2 separate cases, namely $c\neq 0$ and $c=0$. Here we only consider the first as the latter can be done similarly. So now $c\neq 0$, then we get: \begin{align*}
    &(cx+d)^{\deg(F^{Q_G})}h(\frac{ax+b}{cx+d})^{\deg(F)}\\&=(cx+d)^{\deg(F)\deg(Q_G)}\prod\limits_{v\in (G\circ \infty)\setminus\{\infty\}}(\frac{ax+b}{cx+d}-v)^{m_{\infty}\deg(F)}\\&=\frac{(cx+d)^{\deg(F)\cdot |G|}}{(cx+d)^{\deg(F)\cdot(|G\circ \infty|-1)\cdot\frac{|G|}{|G\circ \infty|}}}\prod\limits_{v\in(G\circ \infty)\setminus\{\infty\}}\left((a-cv)x+(b-dv))\right)^{m_{\infty}\cdot\deg(F)}\\&=\left((cx+d)\left(b-d\cdot\frac{a}{c}\right)\prod\limits_{v\in G\circ \infty \setminus\{\infty,\frac{a}{c}\}}(a-cv)\left(x+\frac{b-dv}{a-cv}\right)\right)^{m_{\infty}\cdot\deg(F)}\\&=\alpha_A\left(\left(x+\frac{d}{c}\right)\prod\limits_{v\in G\circ \infty\setminus\{\infty,\frac{a}{c}\}}\left(x-[A^{-1}]\circ v\right)\right)^{m_{\infty}\cdot\deg(F)}\\&=\alpha_A\left(\left(x-[A^{-1}]\circ\infty\right)\prod\limits_{v\in G\circ \infty\setminus\{\infty,\frac{a}{c}\}}\left(x-[A^{-1}]\circ v\right)\right)^{m_{\infty}\cdot\deg(F)}\\&=\alpha_A\left(\prod\limits_{v\in G\circ \infty\setminus\{\frac{a}{c}\}}\left(x-[A^{-1}]\circ v\right)\right)^{m_{\infty}\cdot\deg(F)}\\&=\alpha_A\left(\prod\limits_{u\in G\circ \infty\setminus\{\infty\}}\left(x-u\right)\right)^{m_{\infty}\cdot\deg(F)}=\alpha_A h(x)^{\deg(F)}.
\end{align*} The factor $\alpha_A$ has the form \begin{equation*}
    \alpha_A=\left(\underbrace{c(b-d\cdot\frac{a}{c})}_{=-\det(A)}\prod\limits_{v\in G\circ \infty\setminus\{\infty,\frac{a}{c}\}}(a-cv)\right)^{{m_{\infty}\cdot\deg(F)}},
\end{equation*} so it is non-zero, since all its factors are non-zero. Moreover note that the inverse of $[A]$ in $\PG(K)$ is $[B]$ with \begin{equation*}
    B=\left(\begin{array}{cc}
     d& -b\\
     -c& a
     \end{array}\right),
\end{equation*} because $A\cdot B=\det(A)I_2$. This finishes the proof.
\end{proof}
This lemma does not hold in general if we consider a general generator instead of quotient maps. Consider a quotient map $Q_G=g/h$, then $Q:=Q_G^{-1}=h/g$ is a generator of $K(x)^G$. However, for $F=x$ we get $F^{Q}=h(x)$, which is not in $\NR_K^G$ because $h$ has $G\circ\infty\setminus\{\infty\}$ as its roots; therefore $F^Q$ can not be $G$-invariant. This example suggests that the only exceptions are the monic polynomials $F\in K[x]$ with root $Q(\infty)$. To show that we need a lemma first. \begin{lemma}
Let $Q_G\in K(x)$ be a quotient map for $G$ and $Q\in K(x)$ another generator of $K(x)^G$, then $\deg(Q)=\deg(Q_G)=|G|$ and there is $[C]\in\PG(K)$ such that $Q=[C]\circ Q_G$. More precisely there is \begin{equation*}
    C=\begin{pmatrix}a&b\\c&d\end{pmatrix}\in\operatorname{GL}_2(K) 
\end{equation*} such that  \begin{equation*}
    Q(x)=\frac{aQ_G(x)+b}{cQ_G(x)+d}.
\end{equation*}
\end{lemma}
\begin{proof}
The proof is a combination of Lemma 3.1 and Proposition 3.3 in \cite{bluher1}.
\end{proof}
\begin{corollary}\label{arbgen}
Let $Q_G=\frac{g}{h}$ be a quotient map, $Q\in K(x)$ an arbitrary generator of $K(x)^G$ and $F\in K[x]$ monic. Write $Q=[C]\circ Q_G$, which is possible by the lemma above. If $F([C]\circ \infty)\neq 0$, or equivalently $F(Q(\infty))\neq 0$, then $a\cdot F^{Q}\in\mathcal{NR}_K^G$ and $a\cdot F^Q$ is $G$-invariant (the factor $a\in K^{\ast}$ is needed to make $F^{Q_G}$ monic).
\end{corollary}
\begin{proof}
Let $H=[C]\ast F$, then $\deg(H)=\deg(F)$ by the assumption that $F([C]\circ\infty)\neq 0$ (see Lemma \ref{basic}). Write \begin{equation*}
    C=\begin{pmatrix}a&b\\c&d\end{pmatrix}\in\operatorname{GL}_2(K),
\end{equation*} then $H(x)=\lambda_{C,H}(cx+d)^{\deg(F)}F(\frac{ax+b}{cx+d})$ and $Q(x)=\frac{ag(x)+bh(x)}{cg(x)+dh(x)}$. Moreover \begin{align*}
    H^{Q_G}(x)&=h(x)^{\deg(H)}H(\frac{g(x)}{h(x)})\\&=\lambda_{C,H} h(x)^{\deg(H)}(c\frac{g(x)}{h(x)}+d)^{\deg(F)}F\left(\frac{a\frac{g(x)}{h(x)}+b}{c\frac{g(x)}{h(x)}+d}\right)\\&=\lambda_{C,H}(cg(x)+dh(x))^{\deg(F)}F\left(\frac{ag(x)+bh(x)}{cg(x)+dh(x)}\right)=\lambda_{C,H} F^Q(x).
\end{align*} With Lemma \ref{qginv} we have that $H^{Q_G}$ is $G$-invariant and an element of $\mathcal{NR}_K^G$, thus both facts are also true for $\lambda_{C,H}F^Q$.
\end{proof}
\section{Proof of the Main Theorem}
Here $\overline{K}$ denotes a fixed algebraic closure of $K$ and for a polynomial $P\in K[x]$ we denote its splitting field in $\overline{K}$ as $L_P$. Further we define for an extension $K\subset L\subset \overline{K}$ the set $\operatorname{hom}_K(L)$ as the $K$-automorphisms $\sigma:L\to \overline{K}$. If $L$ is normal over $K$ then $\operatorname{hom}_K(L)=\operatorname{Aut}_K (L)$. We are going to prove the first part of the Main Theorem: 
\begin{theorem}\label{main1}
Let $F\in\mathcal{I}_K$, then there is $k\in\mathbb{N}\setminus\{0\}$ and $r\in\mathcal{I}_K^G$ with $\deg(F)|\deg(r)$ such that \begin{equation*}
    F^{Q_G}=(\prod G\ast r)^k.
\end{equation*} In words: The $Q_G$-transform of an irreducible monic polynomial is a power of an orbit polynomial.
\end{theorem}
\begin{proof}
    We know that $F^{Q_G}\in\NR_K^G$ is $G$-invariant by Lemma \ref{qginv} and therefore all its irreducible factors are contained in $\mathcal{I}_K^G$. First we prove the degree condition for the irreducible factors of $F^{Q_G}$. For that let $r\in\mathcal{I}_K^G$ be an arbitrary irreducible factor of $F^{Q_G}$ and $v\in \overline{K}$ a root of $r$, then $$0= F^{Q_G}(v)=h(v)^{\deg(F)}F(Q_G(v))$$ and $h(v)\neq 0$ because $r\in\NR_K^G$. Thus $F(Q_G(v))=0$ which shows that $Q_G(v)=\alpha\in R_F$ and with that $K(Q_G(v))\subseteq K(v)$. We conclude \begin{equation*}
    \deg(r)=[K(v):K]=[K(v):K(\alpha)]\cdot [K(\alpha):K]=[K(v):K(\alpha)]\cdot \deg(F).
\end{equation*} For the rest note that $G\ast r$ divides $F^{Q_G}$ by Corollary \ref{monoid}. So our goal now is to show that every irreducible factor of $F^{Q_G}$ belongs to $G\ast r$. With Lemma \ref{stich} we know that the set of roots of $F^{Q_G}$ can be partitioned into $G$-orbits under the M\"obius-Transformation, i.e. there exist $v_1,\ldots, v_l\in R_{F^{Q_G}}$ such that $$R_{F^{Q_G}}=\bigcup\limits_{i=1}^l(G\circ v_i)$$ and $(G\circ v_i)\cap (G\circ v_j)=\varnothing$ for $i\neq j$. We set w.l.o.g. $v=v_1$. By Theorem \ref{bij} we know that there are $\alpha_1,\ldots,\alpha_l\in\overline{K}$ such that $Q_G(G\circ v_i)=\alpha_i$. Note that $R_F=\{\alpha_1,\ldots,\alpha_l\}$. Now consider the splitting fields $L_F,L_{F^{Q_G}}$ of $F$ and $F^{Q_G}$ over $K$. The extensions $L_F/K,L_{F^{Q_G}}/K$ and $L_{F^{Q_G}}/L_F$ are normal and finite. It can be shown that for all $\alpha_i,\alpha_j\in R_F$ there is $\sigma_{i,j}\in \operatorname{hom}_K(L_F)=\operatorname{Aut_K}(L_F)$ such that $\sigma_{i,j}(\alpha_i)=\alpha_j$ because $F$ is irreducible (for reference see \cite[Theorem 2.8.3]{RomanS} for example). Now let $\beta\in R_F$ be arbitrary, then there is $\sigma_{\beta}\in\operatorname{Aut}_K(L_F)$ such that $\sigma_{\beta}(\alpha_1)=\beta$. The automorphism $\sigma_{\beta}$ can be extended to an automorphism in $\operatorname{Aut}_K(L_{F^{Q_G}})$, we denote it by $\overline{\sigma}_{\beta}$ (for reference see \cite[Theorem 2.8.4]{RomanS}).
Finally, we put everything together: Let $w\in R_{F^{Q_G}}$ and $Q_G(w)=\gamma\in R_F$, then \begin{align*}
    Q_G(w)=\gamma=\overline{\sigma}_{\gamma}(\alpha)=\overline{\sigma}_{\gamma}(Q_G(v))=Q_G(\overline{\sigma}_{\gamma}(v)),
\end{align*} so $w$ and $\overline{\sigma}_{\gamma}(v)$ are contained in the same $G$-orbit. We just showed that every $G$-orbit in $R_{F^{Q_G}}$ contains at least one root of $r$, since $\sigma(v)$ is always a root of $r$ for all $\sigma\in\operatorname{Aut}_K(L_{F^{Q_G}})$. To finish the proof let $t\in\mathcal{I}_K^G$ be an arbitrary irreducible factor of $F^{Q_G}$ and $w$ a root of $t$, then there is $[A]\in G$ and $v\in R_r$ such that $[A]^{-1}\circ v=w$, thus $t=[A]\ast r$.
\end{proof}
\begin{remark}
This theorem still holds for arbitrary generators $Q=\frac{g}{h}$ of $K(x)^G$ if $F\in \mathcal{I}_K$ satisfies $F(Q(\infty))\neq 0$ because of Corollary \ref{arbgen} with proof. Notice that $Q(\infty)\in K\cup\{\infty\}$, thus for $\deg(F)\ge 2$ it always holds. But $F^Q$ is not guaranteed to be monic, so we have to normalize it on occasion.
\end{remark} What is left to show is that $k=1$ for all but finitely many $F\in\mathcal{I}_K$. The next corollary is very helpful: \begin{corollary}\label{size}
Let $F\in\mathcal{I}_K$, then every $G$-orbit in $R_{F^{Q_G}}$ is of the same size. So for $v\in R_{F^{Q_G}}$ we obtain: \begin{equation*}
    |R_{F^{Q_G}}|=|R_F|\cdot |G\circ v|
\end{equation*}
\end{corollary}
\begin{proof}
    Let $v\in R_{F^{Q_G}}$ be such that $Q_G(G\circ v)=\alpha\in R_F$. Additionally, for $\beta\in R_F$ let $\overline{\sigma}_{\beta}:L_{F^{Q_G}}\to L_{F^{Q_G}}$ be an automorphism of $L_{F^{Q_G}}$ such that $\overline{\sigma}_{\beta}(\alpha)=\beta$ as in the proof of Theorem \ref{main1}. Moreover let $w_{\beta}$ be a root of $F^{Q_G}$ such that $Q_G(w_{\beta})=\beta$. We have $$\overline{\sigma}_{\beta}(G\circ v)\subseteq G\circ w_{\beta}$$ and since $\overline{\sigma}_{\beta}^{-1}\in\operatorname{Aut}_K(L_{F^{Q_G}})$ with $\overline{\sigma}_{\beta}^{-1}(\beta)=\alpha$ also $G\circ w_{\beta}\subseteq \overline{\sigma}_{\beta}(G\circ v)$. Hence $G\circ w_{\beta}= \overline{\sigma}_{\beta}(G\circ v)$ and $|G\circ v|=|\overline{\sigma}_{\beta}(G\circ v)|$ since $\overline{\sigma}_{\beta}$ is bijective on $L_{F^{Q_G}}$. Thus we obtain \begin{align*}
     |R_{F^{Q_G}}|=|\bigcup\limits_{\beta\in R_F}Q_G^{-1}(\beta)|=\sum\limits_{\beta\in R_F}|Q_G^{-1}(\beta)|=|R_F|\cdot |Q_G^{-1}(\alpha)|=|R_F|\cdot |G\circ v|.
\end{align*} 
\end{proof} It follows that if $K$ is perfect, then $F^{Q_G}$ is separable if and only if $G\circ v$ is regular, because then $$|R_{F^{Q_G}}|=|R_F|\cdot |G\circ v|=\deg(F)\cdot |G|=\deg(F^{Q_G}).$$ 
Later we will see that there are only finitely many non-regular $G$-orbits in $\K$ and consequentially there are only finitely many $F\in\mathcal{I}_K$ such that $F^{Q_G}$ is a proper power of an orbit polynomial. But before we do that we want to show that $G\circ v$ is regular for a root of $F^{Q_G}$ implies that $F^{Q_G}$ is a $G$-orbit polynomial and not a proper power thereof holds over every field.
\subsection{Proof of the Second Part of the Main Theorem, Theorem R and Theorem \ref{arbginv}}
Let $L/K$ be a finite field extension. The separable degree of $L$ over $K$ is defined as $$[L:K]_s:=|\operatorname{hom}_K(L)|.$$ Recall that it behaves in the same way as the degree of field extensions, that is, for $M/L/K$ we have $$[M:K]_s=[M:L]_s\cdot [L:K]_s.$$ If $K$ is perfect, then $[L:K]_s=[L:K]$ for every finite field extension $L$ of $K$. Now let $\operatorname{char}(K)=p>0$ and $r\in\mathcal{I}_K$ with root $v\in R_r$. Then there is a natural number $d$ such that $$[K(v):K]=p^d\cdot [K(v):K]_s,$$ this $d$ is called the \textit{radical exponent} of $r$ or $v$ over $K$. It can be shown that the radical exponent of $r$ is the smallest positive integer such that there is an irreducible and separable polynomial $s\in K[x]$ such that $r(x)=s(x^{p^d}).$ For a nice reference on this topic see \cite{RomanS}. For the sake of convenience we use the notation $$\operatorname{rad}(r):=p^d=\frac{[K(v):K]}{[K(v):K)]_s}$$ for $r\in\mathcal{I}_K$ with radical exponent $d$ and $v\in R_r$. 

The essential part of the proof is to show that $\operatorname{rad}(F)=\operatorname{rad}(r)$ for all irreducible factors $r\in\mathcal{I}_K^G$ of $F^{Q_G}$. So let $F\in\mathcal{I}_K$, $r\in\mathcal{I}_K^G$ an irreducible factor of $F^{Q_G}$ and $v\in R_r$ a root of $r$ with $Q_G(v)=\alpha\in R_F$, then \begin{align*}
    \operatorname{rad}(r)&=\frac{[K(v):K]}{[K(v):K]_s}=\frac{[K(v):K(\alpha)]}{[K(v):K(\alpha)]_s}\cdot \frac{[K(\alpha):K]}{[K(\alpha):K]_s}\\&=\frac{[K(v):K(\alpha)]}{[K(v):K(\alpha)]_s}\cdot \operatorname{rad}(F),
\end{align*} so $\operatorname{rad}(F)\le \operatorname{rad}(r)$. For $\operatorname{rad}(F)\ge \operatorname{rad}(r)$ we need to work a bit.
\begin{lemma}\label{Sxq}
Let $F\in\mathcal{I}_K$ be such that $|G\circ v|=|G|$ for $v\in R_{F^{Q_G}}$ and $\operatorname{char}(K)=p>0$. Further $F(x)=H(x^q)$ for $q$ a power of $p$ and $H\in\mathcal{I}_K$ also separable, thus $\operatorname{rad}(F)=q$. Then there is a separable polynomial $S\in K[x]$ such that \begin{equation*}
    F^{Q_G}(x)=S(x^q).
\end{equation*}
\end{lemma}
\begin{proof}
This proof has two parts. At first we show that we can write $F^{Q_G}$ as $S(x^q)$ and afterwards show that the polynomial $S$ is separable. The first part is a calculation exercise. For that let $Q_G=\frac{g}{h}$ and $H=\sum_{i=0}^n a_i x^i$. Additionally, for an arbitrary polynomial $P:=\sum_{i=0}^m c_ix^i$ we define \begin{equation*}
    P^{(q)}:=\sum\limits_{i=0}^m c_i^q x^i.
\end{equation*} Observe that $P(x)^q=P^{(q)}(x^q)$ for $q$ a power of $\operatorname{char}(K)=p$. With that we obtain the following: \begin{align*}
    F^{Q_G}(x)&=h(x)^{\deg(F)}F(\frac{g(x)}{h(x)})=h(x)^{q\deg(H)}H(\frac{g(x)^q}{h(x)^q})\\&=\sum\limits_{i=0}^n a_i g(x)^{iq}h(x)^{(n-i)q}=\sum\limits_{i=0}^n a_i g^{(q)}(x^q)^{i}h^{(q)}(x^q)^{(n-i)}
\end{align*} Since $K[x^q]\subseteq K[x]$ is a subring there exists a polynomial $S$ such that $F^{Q_G}(x)=S(x^q)$, which is exactly what we wanted. The polynomial $S$ is of degree $\deg(S)=\deg(F^{Q_G})/q=|G|\cdot \deg(H)$. For every $v\in R_{F^{Q_G}}$ holds $v^q\in R_S$. The map $y\mapsto y^q$ is bijective on $\overline{K}$, thus $\rho:R_{F^{Q_G}}\to R_S$ with $ v\mapsto v^q$ is injective and therefore $|R_{F^{Q_G}}|\le |R_S|$. Conversely, for $\alpha\in R_S$ the $q$-th root $\alpha^{1/q}$ is a root of $F^{Q_G}$, because $F(\alpha^{1/q})=S((\alpha^{1/q})^q)=S(\alpha)=0$. This shows that $\rho$ is actually a bijection and $|R_S|=|R_{F^{Q_G}}|$. We finish this proof by applying Corollary \ref{size} and using $\deg(H)=|R_F|$ as well as our assumption that $|G|=|G\circ v|$: \begin{equation*}
    |R_S|=|R_{F^{Q_G}}|=|R_F|\cdot |G\circ v|=\deg(H)\cdot |G|=\deg(S)
\end{equation*} So $S$ is separable because it has $\deg(S)$ many roots in $\overline{K}$.
\end{proof}
If $S\in K[x]$ is the separable polynomial such that $F^{Q_G}=S(x^q)$ as in the lemma above, then $S$ factorizes into separable irreducible factors $S=s_1\cdot\ldots\cdot s_l$. Hence $S(x^q)=s_1(x^q)\cdot\ldots\cdot s_l(x^q)$, so it should be beneficial to study the factorization of polynomials of the form $s(x^q)$ for $s\in\mathcal{I}_K$ irreducible and separable. The next lemma shows that such polynomials consist of only one irreducible factor. We give a proof of this lemma as it is a crucial tool for the following theorem. In the proof we employ a similar method as in the proof of Theorem \ref{main1}. We want to point out that there is no finite subgroup $G\le\PG(K)$ with $x^q\in K(x)$ as a quotient map for $q$ a power of $\operatorname{char}(K)$, so this is not a particular case of Theorem \ref{main1}. 
\begin{lemma}\label{facq} Let $s\in K[x]$ be an irreducible, separable and monic polynomial. Furthermore let $\operatorname{char}(K)=p>0$ and $q=p^d$ for $d>0$. Then there is an irreducible,separable and monic polynomial $f\in K[x]$ with $\deg(f)=\deg(s)$ and $a,b\in\mathbb{N}$ with $a+b=d$ such that \begin{equation*}
    s(x^q)=(f(x^{p^a}))^{p^b}
\end{equation*} and $f(x^{p^a})$ is irreducible.
\end{lemma}
\begin{proof}
At first we show that $s(x^q)$ only has one irreducible factor. Since $s$ is both irreducible and separable, $L_s/K$ is a Galois extension. If we set $P(x):=s(x^q)$ and consider the splitting field $L_P/K$ of $P$, then, with similar arguments as in the proof of Theorem \ref{main1}, we can extend every $\sigma\in\operatorname{Gal}(L_s/K)$ to a $K$-homomorphism $\overline{\sigma}\in\operatorname{hom}_K(L_P)$. Since splitting fields are normal, every such $K$-homomorphism is actually an automorphism on $L_P$. Let $F\in\mathcal{I}_K$ be an irreducible factor of $P$ and $v\in R_F$ one of its roots. Then $v^q=:\alpha$ is a root of $s$ and similarly $w^q=:\beta\in R_s$ for $w\in R_P$. Let $\overline{\sigma}\in\operatorname{hom}(L_P)$ be the extension of the homomorphism $\sigma\in\operatorname{Gal}(L_s/K)$ with $\sigma(\alpha)=\beta$, then \begin{equation*}
    w^q=\beta=\overline{\sigma}(\alpha)=\overline{\sigma}(v^q)=(\overline{\sigma}(v))^q.
\end{equation*} As $y\mapsto y^q$ is injective on $\overline{K}$ we obtain that $w=\overline{\sigma}(v)$. Therefore $w$ has to be a root of $F\in K[x]$ as well, since $\overline{\sigma}$ is an automorphism of $L_P$ that fixes $K$. So we just showed that all roots of $P$ are also roots of the irreducible factor $F$, thus $P=F^k$ for $k\in\mathbb{N}$. With the degree formula for field extensions we get \begin{equation*}
    \deg(F)=[K(v):K]=[K(v):K(\alpha)]\cdot[K(\alpha):K]=[K(v):K(\alpha)]\cdot \deg(s),
\end{equation*} which shows $\deg(s)|\deg(F)$. Together with the fact that $k\cdot \deg(F)=\deg(P)=q\cdot \deg(s)$ we get $\deg(F)=p^a\cdot \deg(s)$ for an $a\in\{0,\ldots,d\}$ and $k=p^b$ such that $a+b=d$. Moreover $p^a=[K(v):K(\alpha)]$, which is the degree of the minimal polynomial $m_v\in K(\alpha)[x]$ of $v$ in $K(\alpha)$. Since $v^q=\alpha$, it is also a root of $x^q-\alpha$ and therefore $m_v|x^q-\alpha$. A simple calculation shows that $m_v=x^{p^a}-\alpha^{\frac{1}{p^b}}$, so $\alpha^{\frac{1}{p^b}}\in K(\alpha)$. Conversely $\alpha\in K(\alpha^{\frac{1}{p^b}})$, since $\alpha$ is the $p^b$-th power of $\alpha^{\frac{1}{p^b}}$, hence $K(\alpha)=K(\alpha^{\frac{1}{p^b}})$. Let $f\in \mathcal{I}_K$ be the minimal polynomial of $\alpha^{\frac{1}{p^b}}$ in $K[x]$. We just showed that $\deg(f)=[K(\alpha):K]=\deg(s)$. All that shows that $v$ is also root of the monic polynomial $F':=f(x^{p^a})\in K[x]$. Since $F$ is the minimal polynomial of $v$ over $K$ it has to divide $F'$. Moreover, $F$ and $F'$ have the same degree and are monic, thus have to be equal.
\end{proof}
This is enough to prove that $k=1$ if $G\circ v$ is regular for $v\in R_{F^{Q_G}}$: \begin{theorem}\label{k=1}
Let $F\in \mathcal{I}_K$ be such that $|G\circ v|=|G|$ for $v\in R_{F^{Q_G}}$. Then $F^{Q_G}=\prod(G\ast r)$, i.e. it is an orbit polynomial.
\end{theorem}
\begin{proof} If $F$ is separable and $|G\circ v|=|G|$ for a root of $F^{Q_G}$, then all $G$-orbits in $R_{F^{Q_G}}$ are of the same size by Corollary \ref{size} and as a consequence $F^{Q_G}$ is also separable since $\deg(F^{Q_G})=|R_{F^{Q_G}}|$. If $\operatorname{rad}(F)=q>1$ then 
by Lemma \ref{Sxq} \begin{equation*}
    F^{Q_G}(x)=S(x^q)=\prod\limits_{i=1}^l s_i(x^q)=\prod\limits_{i=1}^l (f_i(x^{p^{a_i}}))^{p^{b_i}},
\end{equation*} where $s_i\in\mathcal{I}_K$ are the irreducible and separable factors of $S$ and $f_i\in\mathcal{I}_K$ the separable and irreducible polynomials as in the lemma above, so $a_i+b_i=d$ for $q=p^d$. Observe that $\gcd(f_i(x^{p^{a_i}}),f_j(x^{p^{a_j}}))=1$ for $i\neq j$. The reason for that is that since $S$ is separable, $s_i$ and $s_j$ have to be different irreducible polynomials, i.e. $\gcd(s_i,s_j)=1$ which is equivalent to $R_{s_i}\cap R_{s_j}=\varnothing$. Now, the roots of $s_i(x^q)$ and $s_j(x^q)$ are the preimages of $R_{s_i}$ and $R_{s_j}$ under the map $y\mapsto y^q$ on $\overline{K}$. This map is bijective and therefore also injective, so the preimages are also different and thus $s_i(x^q)$ and $s_j(x^q)$ have no roots in common. This small observation is very important because now, with the help of Theorem \ref{main1}, we can deduce that \begin{align*}
    S(x^q)&=\prod\limits_{i=1}^l (f_i(x^{p^{a_i}}))^{p^{b_i}}\\
    &=(\prod\limits_{t\in G\ast r} t)^k.
\end{align*} The irreducible factors $f_i(x^{p^{a_i}})$ and $t\in G\ast r$ have to coincide with each other because $K[x]$ is a factorial ring. So we obtain with the remarks above (and the fact that all $t_1,t_2\in G\ast r$ with $t_1\neq t_2$ also satisfy $\gcd(t_1,t_2)=1$ since they are irreducible) that for all $t\in G\ast r$ there is exactly one $i\in [l]$ such that $t(x)=f_i(x^{p^{a_i}})$. This also implies $k=p^{b_i}$ for all $i\in [l]$ and thus $p^{a_i}=p^{a_j}$ for all $i,j\in [l]$. To summarize what we could obtain: \begin{align*}
    S(x^q)=\prod\limits_{i=1}^l (f_i(x^{p^{a}}))^{k}=\prod\limits_{t\in G\ast r} t^k=F^{Q_G}(x). 
\end{align*} This shows that $\operatorname{rad}(t)=p^a\le q=\operatorname{rad}(F)$ for all $t\in G\ast r$, thus $\operatorname{rad}(t)=\operatorname{rad}(F)$, since we already explained that $\operatorname{rad}(r)\ge \operatorname{rad}(F)$. Hence $p^a=q$ and $b_i=0$, which shows $k=1$.
\end{proof} Before we finish the proof of the Main Theorem we want to state an immediate consequence of this result \begin{corollary}\label{maincoroll}
Let $F\in \mathcal{I}_K$ be such that $|G\circ v|=|G|$ for a root $v$ of $F^{Q_G}$, then \begin{enumerate}
    \item The degree of every irreducible factor $r$ of $F^{Q_G}$ satisfies \begin{equation*}
        \deg(r)=\frac{|G|}{|G\ast r|}\cdot \deg(F)
    \end{equation*} 
    \item If $\deg(F)<\deg(r)$ for an irreducible factor $r$ of $F^{Q_G}$, then $\{[I_2]\}\neq\operatorname{Stab}_G(r)\le G$ is non-trivial
\end{enumerate}
\end{corollary}
\begin{proof}
For the first we use Theorem \ref{k=1} (so $k=1$) and calculate: \begin{equation*}
    |G|\cdot \deg(F)=\deg(F^{Q_G})=|G\ast r|\cdot \deg(r).
\end{equation*} The last is obvious because $\deg(r)=\frac{|G|}{|G\ast r|}\cdot\deg(F)>\deg(F)$, so $|\operatorname{Stab}_G(r)|=\frac{|G|}{|G\ast r|}>1$.
\end{proof}
\begin{remark}
Observe that for $Q$ an arbitrary generator of $K(x)^G$ Theorem \ref{k=1} and Corollary \ref{maincoroll} still hold if $F\in\mathcal{I}_K$ and $F(Q(\infty))\neq 0$. The reason for this is again the proof of Corollary \ref{arbgen}: Write 
\begin{equation*}
    Q(x)=[C]\circ Q_G(x)=\frac{aQ_G(x)+b}{cQ_G(x)+d}
\end{equation*} as in Corollary \ref{arbgen}, where $Q_G$ is a quotient map, then there is $H\in\mathcal{I}_K$ such that $H^{Q_G}=a\cdot F^Q$. Since we always assume something about the roots of the resulting polynomial, i.e. about $F^Q$ and thus also about $H^{Q_G}$, both Theorem \ref{k=1} and Corollary \ref{maincoroll} hold because if $F^Q$ only contains roots in regular $G$-orbits so does $H^{Q_G}$. 
\end{remark}
We state an analogue of Theorem \ref{bij} for polynomials: \begin{corollary}\label{main2}
The map $\delta_{Q_G}:\mathcal{I}_K\to\mathcal{I}_K^G/G$ with $F\mapsto G\ast r$ such that $F^{Q_G}=\prod(G\ast r)^k$ is a bijection.  
\end{corollary}
\begin{proof}
By Theorem \ref{main1} $\delta_{Q_G}$ defines a mapping between $\mathcal{I}_K$ and $\mathcal{I}_K^G/G$. First we show that $\delta_{Q_G}$ is surjective. Let $r\in\mathcal{I}_K^G$ and $v\in R_r$ a root of $r$. Then $r$ is the minimal polynomial of $v$ over $K$. Moreover, let $\alpha\in \overline{K}$ be such that $Q_G(v)=\alpha$ and denote by $F\in \mathcal{I}_K$ the minimal polynomial of $\alpha$. Then $F^{Q_G}$ has $v$ as a root, thus $r|F^{Q_G}$ and $F^{Q_G}=(\prod(G\ast r))^k$ by Theorem \ref{main1}, so $\delta_{Q_G}(F)=G\ast r$. Now onto the injectivity: Let $F,H\in\mathcal{I}_K$ be such that $\delta_{Q_G}(F)=\delta_{Q_G}(H)=G\ast r$ for $r\in\mathcal{I}_K^G$, so $F^{Q_G}=(\prod(G\ast r))^k$ and $H^{Q_G}=(\prod(G\ast r))^l$ and both $F^{Q_G}$ and $H^{Q_G}$ have the same roots. With the help of what we observed in the proof of Theorem \ref{main1} we obtain: \begin{equation*}
    \bigcup\limits_{\alpha\in R_F}Q_G^{-1}(\alpha)=R_{F^{Q_G}}=R_{H^{Q_G}}=\bigcup\limits_{\beta\in R_H}Q_G^{-1}(\beta).
\end{equation*} Therefore $R_F=R_H$ since $Q_G$ induces a bijection between $\K$ and $(\K)/G$ by Theorem \ref{bij}. As $F$ and $H$ are irreducible, monic and share the same roots they have to be equal. 
\end{proof} As an immediate consequence of this corollary we obtain the main part of the general version of Theorem R: \begin{theorem}\label{genreis} If $f\in\mathcal{I}_K^G$ is a $G$-invariant monic irreducible polynomial with root $v\in R_f$ that is contained in a regular $G$-orbit, then there is $F\in \mathcal{I}_K$ such that $f=F^{Q_G}$.
\end{theorem}
\begin{proof}
If $f\in\mathcal{I}_K^G$ is $G$-invariant, then $G\ast f=\{f\}$. By Corollary \ref{main2} we get that there is $F\in\mathcal{I}_K$ such that $\delta_{Q_G}(F)=\{f\}$, which translates to $F^{Q_G}=f^k$. Further $k=1$ because of Theorem \ref{k=1} and the assumption that $v\in R_f$ is contained in a regular $G$-orbit. 
\end{proof}
To complete the proofs of the Main Theorem and the general Theorem R we need to show that the number of irreducible polynomials for which $F^{Q_G}=(\prod G\ast r)^k$ with $k>1$ is finite. By Theorem \ref{k=1} this is equivalent to showing that the number of irreducible polynomials $F\in\mathcal{I}_K$ for which $F^{Q_G}$ has roots in non-regular $G$-orbits is finite. We give these polynomials the following name: \begin{definition}
    We call $F\in\mathcal{I}_K$ \textbf{$Q_G$-non-conformal} if $F^{Q_G}=(\prod(G\ast r))^k$ for $r\in\mathcal{I}_K^G$, $k\in\mathbb{N}$ and $k>1$. For the set of $Q_G$-non-conformal polynomials we write $\mathcal{NC}^{Q_G}$.
\end{definition} Further we set \begin{equation*}
    P_G:=\{u\in \overline{K}\cup\{\infty\}|~|G\circ u|<|G|\}
\end{equation*} as the set of elements in $\K$ contained in non-regular $G$-orbits. We have the following: \begin{lemma}[{\cite[Lemma 2.1]{bluher1}}]\label{deg2}
Let $G\le\PG(K)$ be finite and $v\in \overline{K}\cup\{\infty\}$. Then $G\circ v$ is non-regular if and only if there is $[A]\in G\setminus\{[I_2]\}$ such that $[A]\circ v=v$, thus \begin{equation*}
    P_G=\{u\in \overline{K}\cup\{\infty\}|~\exists [A]\in G\setminus\{[I_2]\}:~[A]\circ u=u\}
\end{equation*} and this set is finite; more precisely $|P_G|\le 2(|G|-1)$. Furthermore $[K(u):K]\le 2$ for all $u\in P_G\setminus\{\infty\}$. 
\end{lemma}
We denote by $\mathcal{P}_G$ the set of minimal polynomials of elements in $P_G$, that is, \begin{equation}\label{shoeshpg}
    \mathcal{P}_G:=\{r\in\mathcal{I}_K^G|\exists \alpha\in P_G:~ r(\alpha)=0\}.
\end{equation} Notice that $\mathcal{P}_G\subseteq \mathcal{I}_K^G$ by definition, thus $\mathcal{P}_G$ does not contain polynomials with roots in $(G\circ \infty)\setminus\{\infty\}$, even if this orbit is non-regular. We obtain
\begin{lemma}\label{nonconformlem}
We have\footnote{The set $\mathcal{NC}^{Q_G}$ can be empty. In that case we define the left side of the subset equation to be empty as well.} \begin{equation*}
    \bigcup\limits_{F\in\mathcal{NC}^{Q_G}}\delta_{Q_G}(F)\subseteq \mathcal{P}_G.
\end{equation*} 
 In particular \begin{equation*}
    |\mathcal{NC}^{Q_G}|\le |\mathcal{P}_G|\le |P_G|\le 2(|G|-1),
\end{equation*} so there are only finitely many $Q_G$-non-conformal polynomials.
\end{lemma}
\begin{proof}
The map $\mathcal{NC}^{Q_G}\to\mathcal{P}_G/G$ with $F\mapsto \delta_{Q_G}(F)$ defines an injective mapping by Theorem \ref{main2} and Theorem \ref{k=1}, thus the subset equation \begin{equation*}
    \bigcup\limits_{F\in\mathcal{NC}^{Q_G}}\delta_{Q_G}(F)\subseteq \mathcal{P}_G
\end{equation*} follows. Since $\mathcal{P}_G$ only contains minimal-polynomials of elements in $P_G$ and the degree of $r\in\mathcal{P}_G$ is either 1 or 2 by Lemma \ref{deg2} we get that $|\mathcal{P}_G|\le |P_G|$, so $\mathcal{P}_G$ is finite because $|P_G|\le 2(|G|-1)$ by Lemma \ref{deg2}. Hence $\mathcal{NC}^{Q_G}$ is finite as well and \begin{equation*}
    |\mathcal{P}_G|\ge \sum\limits_{F\in\mathcal{NC}^{Q_G}}|\delta_{Q_G}(F)|\ge \sum\limits_{F\in\mathcal{NC}^{Q_G}}1=|\mathcal{NC}^{Q_G}|.
\end{equation*}
\end{proof} We close this section by proving Theorem \ref{arbginv}: \begin{theorem}\label{fac2}
Let $F_1,\ldots F_l\in\mathcal{NC}^{Q_G}$ be all $Q_G$-non-conformal polynomials. Further let $r_1,\ldots,r_l\in\mathcal{I}_K^G$ be such that $\delta_{Q_G}(F_i)=G\ast r_i$ and $n_i\in\mathbb{N}$ such that $F_i^{Q_G}=(\prod G\ast r_i)^{n_i}$. Then for every $G$-invariant polynomial $f\in\mathcal{NR}_K^G$ there is a unique monic polynomial $F\in K[x]$ and unique natural numbers $0\le k_i<n_i$ such that \begin{equation*}
    f=\left(\prod\limits_{i=1}^l(\prod\limits G\ast r_i)^{k_i}\right)\cdot F^{Q_G}.
\end{equation*}
\end{theorem}
\begin{proof}
By Corollary \ref{fac1} we have \begin{equation*}
    f=\prod\limits_{i=1}^e(\prod(G\ast g_i))^{m_i}
\end{equation*} with $g_i\in\mathcal{I}_K^G$ and $m_1,\ldots,m_e\in\mathbb{N}\setminus\{0\}$. We refine this factorization by grouping the orbit polynomials into either belonging to $\mathcal{P}_G$ or not, which gives \begin{equation*}
    f=\prod\limits_{i=1}^l(\prod(G\ast r_i))^{l_i}\cdot \prod\limits_{j=1}^c(\prod(G\ast h_j))^{d_j};
\end{equation*} here we allow $l_i=0$. Since $h_j\in\mathcal{I}_K^G\setminus\mathcal{P}_G$ for all $j\in [c]$ there exists a unique monic polynomial $H_1\in K[x]$ such that
\begin{equation*}
    H_1^{Q_G}=\prod\limits_{j=1}^c(\prod(G\ast h_j))^{d_j}
\end{equation*} by Theorem \ref{k=1} together with Lemma \ref{basicqg} item 2. For the remaining factor we divide $l_i$ by $n_i$ and write $k_i$ for the remainder, so $l_i=a_i\cdot n_i+k_i$ and $0\le k_i<n_i$. We have
\begin{equation*}
    (F_i^{a_i})^{Q_G}=(\prod(G\ast r_i))^{a_i\cdot n_i}.
\end{equation*} Hence we define \begin{equation*}
    H_2:=\prod_{i=1}^l F_i^{a_i}
\end{equation*} and get \begin{equation*}
    f=\left(\prod\limits_{i=1}^l(\prod\limits G\ast r_i)^{k_i}\right)\cdot(H_2\cdot H_1)^{Q_G}.
\end{equation*} Set $F=H_2\cdot H_1$, which is unique since both $H_1$ and $H_2$ are unique.
\end{proof}
\section{Some Notes on the Galois Theory of Invariant Polynomials}
In this section we want to explain some statements about the Galois theory of $G$-invariant polynomials and their implications for finite fields. In particular, we give an alternative proof of the fact that all irreducible monic polynomials $f\in\F_q[x]$ of degree $\deg(f)\ge 3$ have cyclic stabilizers in $\PG(\F_q)$ (\cite[Theorem 1.3]{reisEx}). This means that if $f\in\mathcal{I}_{\F_q}=:\mathcal{I}_q$ is of degree $\deg(f)\ge 3$ and $G\le \PG(\F_q)$ is non-cyclic, then $|G\ast f|>1$. We will exploit this in the proof of Theorem \ref{facfin}.

 Consider the $G$-invariant separable and monic polynomial $f\in\mathcal{I}_K^G$ with roots belonging to regular $G$-orbits and its splitting field $L_f$ in a fixed algebraic closure of $K$. As seen before we can partition the set of roots of $f$ into $G$-orbits \begin{equation*}
    R_f=\bigcup_{i=1}^k(G\circ v_i)
\end{equation*} where $v_1,\ldots v_k\in R_f$ is a set of representatives of $R_f/G$. Thus $\deg(f)=|G|\cdot k$, since all $G$-orbits in $R_f$ are regular. Let $\sigma\in\operatorname{Gal}(f):=\operatorname{Gal}(L_f/K)$ and \begin{equation*}
    A=\begin{pmatrix}a&b\\c&d\end{pmatrix}
\end{equation*} such that $[A]\in G$, then \begin{equation*}
    \sigma([A]\circ v)=\sigma(\frac{av+b}{cv+d})=\frac{a\sigma(v)+b}{c\sigma(v)+d}=[A]\circ \sigma(v)
\end{equation*} for all $v\in R_f$. This shows that the actions of $G$ and $\operatorname{Gal}(f)$ on $R_f$ commute and that $G\circ v_1,\ldots,G\circ v_k$ is a non-trivial block system for $\operatorname{Gal}(f)$: \begin{definition}[See \cite{permutg}]
    Let $G$ be a finite group acting transitively on a non-empty finite set $X$. We say that a subset $Y\subseteq X$ is a block for $G$ if $g\cdot Y=Y$ or $g\cdot Y\cap Y=\varnothing$. Moreover, $Y$ is a non-trivial block if $1<|Y|<|X|$. If $Y$ is a block then $\{g\cdot Y|g\in G\}$ is a partition of $X$ and is called a block system of $X$ for $G$. 
\end{definition}
We define the point-wise stabilizer subgroup of a block $G\circ v$ as\begin{equation*}
    \operatorname{p-Stab}_{\operatorname{Gal}(f)}(G\circ v):=\{\sigma\in \operatorname{Gal}(f)|\sigma(w)=w \text{ for all }w\in G\circ v\}.
\end{equation*} Similarly, the set-wise stabilizer is  \begin{equation*}
    \operatorname{s-Stab}_{\operatorname{Gal}(f)}(G\circ v):=\{\sigma\in \operatorname{Gal}(f)|\sigma(w)\in G\circ v \text{ for all }w\in G\circ v\}.
\end{equation*} Notice that $ \operatorname{p-Stab}_{\operatorname{Gal}(f)}(G\circ v)\trianglelefteq \operatorname{s-Stab}_{\operatorname{Gal}(f)}(G\circ v)$. Our first goal is to show that $G$ is isomorphic to the quotient of these stabilizers. For that, we need a nice lemma about commuting group actions stated in \cite{guire} and \cite{Gow}: \begin{lemma}[\cite{guire} \& \cite{Gow}] \label{commuting}
Let $X$ be a finite non-empty set and $G,H$ groups acting transitively and faithful\footnote{$G$ acts faithful on $X$ if $g\cdot x=x$ for all $x\in X$ implies $g=1$} on $X$. Moreover, $G$ acts regularly on $X$, that is, $\operatorname{Stab}_G(x)=\{1\}$ for all $x\in X$ and the actions of $G$ and $H$ commute, i.e. \begin{equation*}
    g(h(x))=h(g(x))
\end{equation*} for all $h\in H$, $g\in G$ and $x\in X$. Then $H$ acts regularly on $X$ and is isomorphic to $G$.
\end{lemma} Additionally we prove the following \begin{lemma}
Let $f\in\mathcal{I}_K^G$ be $G$-invariant, separable and $R_f$ only contains regular $G$-orbits. Moreover let $v \in R_f$ be a root of $f$. Then we have: \begin{enumerate}
    \item $G$ acts transitively and regularly on $G\circ v$
    \item $U:=\operatorname{s-Stab}_{\operatorname{Gal}(f)}(G\circ v)/\operatorname{p-Stab}_{\operatorname{Gal}(f)}(G\circ v)$ acts faithful and transitively on $G\circ v$
    \item $\operatorname{p-Stab}_{\operatorname{Gal}(f)}(G\circ v)=\operatorname{Stab}_{\operatorname{Gal}(f)}(v)$, thus $U$ acts regularly on $G\circ v$
\end{enumerate}
\end{lemma}
\begin{proof}
For the first item note that the action of $G$ restricted to any of its orbits in $R_f$ is always transitive. Additionally, $G$ acts regularly on $G\circ v$ since all orbits are regular. That the induced action of $U$ on $G\circ v$ is faithful and transitive follows from standard facts about group actions, so onto the last item: The inclusion $\operatorname{p-Stab}_{\operatorname{Gal}(f)}(G\circ v)\subseteq \operatorname{Stab}_{\operatorname{Gal}(f)}(v)$ is obvious. For the other let $\sigma\in \operatorname{Stab}_{\operatorname{Gal}(f)}(v)$, so $\sigma(v)=v$. Moreover, by the first item, $G$ acts transitively on $G\circ v$, so for all $w\in G\circ v$ there is $[A]\in G$ such that $[A]\circ v=w$. As a consequence \begin{equation*}
    \sigma(w)=\sigma([A]\circ v)=[A]\circ \sigma(v)=[A]\circ v=w,
\end{equation*} so both sets are equal. Moreover, all stabilizers of elements in $G\circ v$ are equal to $\operatorname{p-Stab}_{\operatorname{Gal}(f)}(G\circ v)$. So for all  $w\in G\circ v$ we get \begin{equation*}
    U=\operatorname{s-Stab}_{\operatorname{Gal}(f)}(G\circ v)/\operatorname{Stab}_{\operatorname{Gal}(f)}(w),
\end{equation*} thus $U$ acts regularly on $G\circ v$.
\end{proof}
We apply Lemma \ref{commuting} to our setup. We set $G$ as $G$ in Lemma \ref{commuting} and $H=U$ as in the previous lemma, then \begin{equation}\label{firstiso}
    G\cong U=\operatorname{s-Stab}_{\operatorname{Gal}(f)}(G\circ v)/\operatorname{Stab}_{\operatorname{Gal}(f)}(v).
\end{equation} With that we can obtain \begin{corollary}\label{galeasy}
Let $f\in\mathcal{I}_K^G$ be $G$-invariant, separable and all $G$-orbits in $R_f$ are regular. Moreover let $v \in R_f$ be a root of $f$. \begin{enumerate}
    \item If $\operatorname{Gal}(f)$ is abelian, then $G\cong U\le \operatorname{Gal}(f)$
    \item If $\deg(f)=|G|$, then $G\cong \operatorname{Gal}(f)$
\end{enumerate}
\end{corollary}
\begin{proof}
We want to show that $\operatorname{Gal}(f)$ is abelian implies $\operatorname{Stab}_{\operatorname{Gal}(f)}(v)=\{\operatorname{id}\}$. To see this let $\sigma\in \operatorname{Stab}_{\operatorname{Gal}(f)}(v)$ and $w\in R_f$. Additionally set $\tau\in\operatorname{Gal}(f)$ such that $\tau(v)=w$ (exists since $f$ is irreducible and thus $\operatorname{Gal}(f)$ acts transitively on $R_f$). Then we obtain\begin{equation*}
    \sigma(w)=\sigma(\tau(v))=\tau(\sigma(v))=\tau(v)=w,
\end{equation*} so $\sigma(w)=w$ for all $w\in R_f$ and thus $\sigma=\operatorname{id}$ because $\operatorname{Gal}(f)$ acts faithful on $R_f$. Consequentially $G\cong \operatorname{s-Stab}_{\operatorname{Gal}(f)}(G\circ v)\le \operatorname{Gal}(f)$, which finishes the first part.

Now let $f$ be of degree $|G|$, so also $|R_f|=|G|$ and $R_f=G\circ v$ for all $v\in R_f$. Therefore $\operatorname{Gal}(f)=\operatorname{s-Stab}_{\operatorname{Gal}(f)}(G\circ v)$ by definition and $\operatorname{p-Stab}_{\operatorname{Gal}(f)}(G\circ v)=\{\operatorname{id}\}$ because $\operatorname{Gal}(f)$ acts faithful on $R_f$. This shows \begin{equation*}
    \operatorname{Gal}(f)=\operatorname{s-Stab}_{\operatorname{Gal}(f)}(G\circ v)\cong G.
\end{equation*} 
\end{proof}
Further, we get:
\begin{corollary}\label{ffgal}
Let $f\in\mathcal{I}_q$. If $f$ is $G$-invariant for a subgroup $G\le \PG(\F_q)$ and $R_f$ only contains regular $G$-orbits then $G$ has to be cyclic and $|G|\mid \deg(f)$.
\end{corollary}
\begin{proof}
It is well-known that $\operatorname{Gal}(f)\cong C_{\deg(f)}$, where $C_n$ is the cyclic group of order $n$, so $\operatorname{Gal}(f)$ is also abelian. If $f$ is $G$-invariant for $G\le \PG(\F_q)$, then $G$ has to be isomorphic to a subgroup of $C_{\deg(f)}$ by Corollary \ref{galeasy} (1) and therefore has to be cyclic as well. By Lagrange's theorem $|G|$ divides $|C_{\deg(f)}|=\deg(f)$. 
\end{proof}
\begin{remark}
   This result does not hold for quadratic irreducible polynomials over finite fields. Define $\mathcal{I}_q^n$ as the set of monic irreducible polynomials over $\F_q$ of degree $n$. It can be shown that for $g\in\mathcal{I}_q^2$ we have $$G\ast g=\mathcal{I}_q^2$$ and thus $$|\operatorname{Stab}_{\PG(\F_q)}(g)|=\frac{|\PG(\F_q)|}{|\mathcal{I}_q^2|}=2(q+1).$$ Since the biggest order of an element in $\PG(\F_q)$ is $q+1$ the stabilizer can not be cyclic. In fact, it is dihedral. The reason why Corollary \ref{ffgal} fails is that the set of roots of a quadratic irreducible polynomial $g\in\mathcal{I}_q^2$ does not contain regular $\operatorname{Stab}_{\PG(\F_q)}(g)$-orbits.
\end{remark} 
We have enough to prove Theorem \ref{facfin}
\begin{proof}[Proof. (Theorem \ref{facfin})] Since $F^{Q_G}$ is separable it has degree many roots and together with Corollary \ref{size} we have that all $G$-orbits in $R_{F^{Q_G}}$ are regular. Note that $K=\F_ q$ is perfect so $|R_r|=\deg(r)$ for all irreducible polynomials in $\F_q[x]$.
    With Theorem \ref{main1} and \ref{k=1} we know that there exists an irreducible polynomial $r\in \mathcal{I}_K^G$ such that $F^{Q_G}=\prod G\ast r$. Observe that $r$ is $\operatorname{Stab}_G(r)\le G$ invariant and the $\operatorname{Stab}_G(r)$-orbits in $R_r$ are regular because the $G$-orbits in $R_{F^{Q_G}}$ are regular.  Consequentially $\operatorname{Stab}_G(r)$ has to be cyclic by Corollary \ref{ffgal} and  with Corollary \ref{maincoroll} we obtain \begin{equation*}
    \deg(r)=|\operatorname{Stab}_G(r)|\cdot \deg(F)\le \mu_G\cdot \deg(F)
\end{equation*} and \begin{equation*}
    |G\ast r|=\frac{|G|}{|\operatorname{Stab}_G(r)|}\ge \frac{|G|}{\mu_G}.
\end{equation*}
\end{proof}
\begin{corollary}
    If $G\le \PG(\F_q)$ is non-cyclic and $Q_G\in\F_q(x)$ is a quotient map for $G$. Then $F^{Q_G}$ is reducible for all $F\in \F_q[x]$.
\end{corollary}
\begin{proof}
    Note that ''$F^{Q_G}$ is irreducible'' implies ''$F$ is irreducible'', so we can focus on $F$ being irreducible. If $F\in \mathcal{I}_K\setminus \mathcal{NC}^{Q_G}$, then $F^{Q_G}$ has at least $|G|/\mu_G$ irreducible factors by Theorem \ref{facfin} and $|G|/\mu_G>1$ if $G$ is non-cyclic. If $F\in\mathcal{NC}^{Q_G}$, then \begin{equation*}
        F^{Q_G}=\prod(G\ast r)^k
    \end{equation*} for $r\in\mathcal{I}_K^G$ and $k>1$.
\end{proof}
\section{Examples of Invariant Polynomials} In this section we show how our result apply to specific subgroups of $\PG(K)$ where $\operatorname{char}(K)>0$. 
\subsection{Unipotent Subgroups} Consider a field $K$ with $\operatorname{char}(K)=p>0$ and $q=p^l$ for $l>0$. Moreover assume $\F_q\subseteq K$ and let $V\le_q K$ be a $\F_q$-subspace of $K$ of dimension $n\in\mathbb{N}\setminus\{0\}$. For the subspace $V$ we define \begin{equation*}
    \overset{\sim}{V}:=\left\{\left[\left(\begin{array}{cc}
     1& v\\
     0& 1
     \end{array}\right)\right]: v\in V\right\}\le\PG(K).
\end{equation*} Observe that $\overset{\sim}{V}\cong \F_q^n$ as groups, so $\overset{\sim}{V}$ is abelian and every non-trivial element $[A]\in \overset{\sim}{V}$ has order $p$. Additionally, $\overset{\sim}{V}\subseteq \operatorname{Stab}_{\PG(K)}(\infty)$, so $\NR_K^{\overset{\sim}{V}}$ is just the set of monic polynomials over $K$. A quotient map is the to $V$ associated subspace polynomial (see \cite[§10]{bluher1}) \begin{equation}
    Q_{\overset{\sim}{V}}(x)=\prod_{v\in V}(x-v).
\end{equation} The set $P_{\overset{\sim}{V}}$ only contains $\infty$, so $\mathcal{P}_{\overset{\sim}{V}}=\varnothing=\mathcal{NC}^{Q_{\overset{\sim}{V}}}$. This makes the classification of $\overset{\sim}{V}$-invariant polynomials especially nice: \begin{corollary}
For every monic $\overset{\sim}{V}$-invariant polynomial $f\in \NR_K^{\overset{\sim}{V}}$ exists a unique monic polynomial $F\in K[x]$ such that \begin{equation*}
    f(x)=F\left(\prod_{v\in V}(x-v)\right).
\end{equation*} 
\end{corollary}
\begin{remark}
This result is already known, see \cite[Theorem 2.5.]{ReisGLq}. There, $\overset{\sim}{V}$-invariant polynomials are called $V$-translation invariant polynomials since \begin{equation*}
    [A]\ast f(x)=f(x+v)
\end{equation*} for \begin{equation*}
    A=\left(\begin{array}{cc}
     1& v\\
     0& 1
     \end{array}\right).
\end{equation*} Even though the proof of Theorem 2.5. is only stated for finite fields, the assumption that $K$ is finite is not used at all so it also holds if $K$ is infinite. We gave an alternative proof of this result.
\end{remark}
Next we look at the factorization of $F(Q_{\overset{\sim}{V}}(x))$.
\begin{lemma}\label{unipotentlem}
Let $F\in\mathcal{I}_K$, then we obtain: \begin{enumerate}
    \item All irreducible factors of $F(Q_{\overset{\sim}{V}}(x))$ have the same stabilizer in $\overset{\sim}{V}$, i.e. there is $W\le V$ such that all irreducible factors of $F(Q_{\overset{\sim}{V}}(x))$ are $\overset{\sim}{W}$-invariant.
    \item Let $F(Q_{\overset{\sim}{V}}(x))=\prod(\overset{\sim}{V}\ast r)$ for an $r\in \mathcal{I}_K$ and $\operatorname{Stab}_{\overset{\sim}{V}}(r)=\overset{\sim}{W}$ for $W\le V$. Moreover let $v_1,\ldots, v_k\in V$ be a complete set of representatives for $V/W$, then \begin{equation}\label{cosetfac}
        F(Q_{\overset{\sim}{V}}(x))=\prod\limits_{i=1}^k r(x+v_i). 
    \end{equation}
\end{enumerate} 
\end{lemma}
\begin{proof}
By Theorem \ref{main1} and \ref{k=1} we know that $F(Q_{\overset{\sim}{V}}(x))=\prod (\overset{\sim}{V}\ast r)$ for an $r\in\mathcal{I}_K^{\overset{\sim}{V}}=\mathcal{I}_K$. All elements in the same orbit have conjugated stabilizers. Since $\overset{\sim}{V}$ is abelian every subgroup is normal, thus $\operatorname{Stab}_{\overset{\sim}{V}}(t)=\operatorname{Stab}_{\overset{\sim}{V}}(r)$ for all $t\in \overset{\sim}{V}\ast r$, so the first part is proved.

For the second item notice that for $v,u\in V$ we have \begin{equation*}
   r(x+v)=r(x+u) \Leftrightarrow r(x)=r(x+(u-v)) \Leftrightarrow u-v\in W,
\end{equation*} hence all $r(x+v_i)$ are different irreducible polynomials for $v_1,\ldots,v_k$ a complete set of representatives of $V/W$. Since $F(Q_{\overset{\sim}{V}}(x))$ has exactly $|V|/|W|$ irreducible factors by Corollary \ref{maincoroll} equation (\ref{cosetfac}) follows. 
\end{proof}
Note that Corollary \ref{facunipot} is an immediate consequence of Theorem \ref{facfin}.
\subsection{Borel-Subgroup}
Here we consider the \textit{Borel}-subgroup of $\PG(q)$ in fields $K$ with $\F_q\subseteq K$. This groups is defined as \begin{equation*}
    B(q)=\left\{\left[\left(\begin{array}{cc}
     a& b\\
     0& 1
     \end{array}\right)\right]: a\in\F_q^{\ast},b\in\F_q\right\}.
\end{equation*} The transformation with $[A]\in B(q)$ looks like  \begin{equation*}
    [A]\ast f(x)=a^{-\deg(f)}\cdot f(ax+b)
\end{equation*} for \begin{equation}\label{bq}
    A=\left(\begin{array}{cc}
     a& b\\
     0& 1
     \end{array}\right).  
\end{equation} The group can be seen as \begin{equation*}
    B(q)\cong \F_q\rtimes\F_q^{\ast}
\end{equation*} where the multiplication is defined as $$(b,a)\cdot (d,c)=(b+ad,ac)$$ and for $A\in\operatorname{GL}_2(K)$ as in (\ref{bq}) we have \begin{equation*}
    \operatorname{ord}([A])=\begin{cases}
    \operatorname{ord}_{\F_q^{\ast}}(a),& \text{ if } a\neq 1\\
    p,& \text{ if }a=1\text{ and } b\neq 0
    \end{cases}
\end{equation*} Moreover $B(q)=\operatorname{Stab}_{\PG(q)}(\infty)$, so $\NR_K^{B(q)}$ is the set of monic polynomials in $\F_q[x]$. We calculate a quotient map for $B(q)$ using \begin{lemma}[{\cite[Theorem 3.10]{bluher1}}]\label{guide}
Let $G\le\PG(K)$ be a finite subgroup and  for $v\in \overline{K}$ let \begin{equation*}
    g_v(x):=\prod\limits_{u\in G\circ v}(x-u)^{m_v} \text{ and } h_{\infty}(x)=\prod\limits_{u\in G\circ \infty\setminus\{\infty\}}(x-u)^{m_{\infty}},
\end{equation*}
where
$m_u:=|\operatorname{Stab}_G(u)|$ for $u\in\overline{K}\cup\{\infty\}$. Then there is $w\in \overline{K}$ such that \begin{equation*}
    Q_G(x):=\frac{g_v(x)}{h_{\infty}(x)}+w\in K(x)
\end{equation*}
is a quotient map for $G$. Conversely, if there is $w\in \overline{K}$ with $\frac{g_v(x)}{h_{\infty}(x)}+w\in K(x)$, then $Q(x):=\frac{g_v(x)}{h_{\infty}(x)}+w$ is a quotient map for $G$.
\end{lemma}
Let $v\in\F_{q^2}\setminus\F_q$, then $B(q)\circ v=\F_{q^2}\setminus\F_q$ because $\{1,v\}$ is a $\F_q$ basis of $\F_{q^2}$, thus $$\{a\cdot v+b|a\in\F_q^{\ast},b\in\F_q\}=\F_{q^2}\setminus\F_q.$$ Hence, a quotient map is given by \begin{align*}
    Q_{B(q)}(x)&=\prod\limits_{w\in B(q)\circ v}(x-w)^{|\operatorname{Stab}_{B(q)}(v)|}=\prod\limits_{w\in\F_{q^2}\setminus\F_q}(x-w)^{|\operatorname{Stab}_{B(q)}(v)|}\\&=\prod\limits_{g\in\mathcal{I}_q^2}g(x).
\end{align*}Recall that $\mathcal{I}_q^2$ is the set of monic irreducible polynomials of degree $2$ over $\F_q$. That $|\operatorname{Stab}_{B(q)}(v)|=1$ holds is a consequence of $av+b=v$ only having solutions $v\in\F_q\cup\{\infty\}$ if $(a,b)\neq (1,0)$. 

If $q\neq 2$ then $P_{B(q)}$ is equal to $\F_q\cup\{\infty\}$ because $A$ as in (\ref{bq}) fixes $-b/(a-1)$. For $q=2$ we have $B(2)\cong \F_2$ and $P_{B(2)}=\{\infty\}$; so this case belongs to the previous example. The group $B(q)$ acts transitively on $P_G\setminus\{\infty\}=\F_q$ because for fixed $c\in\F_q$ and $b\in\F_q$ arbitrarily take the matrix \begin{equation*}
    B=\left(\begin{array}{cc}
     1& b-c\\
     0& 1
     \end{array}\right)
\end{equation*} and we see $[B]\circ c=c+(b-c)=b$. Consequentially, $B(q)$ also acts transitively on $\mathcal{P}_G$, which consists of polynomials of the form $x-c$ for $c\in \F_q$. Hence there is a monic irreducible polynomial $F$ of degree $1$ (so $F=x-\alpha$ for $\alpha\in K$) and an exponent $k>1$ such that \begin{equation*}
   \left(\prod\limits_{g\in\mathcal{I}_q^2}g(x)\right)-\alpha= F^{Q_{B(q)}}(x)=\left(\prod\limits_{v\in\F_q}(x-v)\right)^k=(x^q-x)^k.
\end{equation*} We can deduce the exponent $k$ from comparing the degree of the polynomials on both sides of the equality. The left side has degree $2\cdot N_q(2)=2\cdot(\frac{1}{2}(q^2-q))=q^2-q=q(q-1)$, so $k=q-1$. To obtain $\alpha$ we calculate \begin{align*}
    (x^q-x)^{q-1}=\frac{x^{q^2}-x^q}{x^q-x}=\frac{x^{q^2}-x}{x^q-x}-\frac{x^q-x}{x^q-x}=\prod\limits_{g\in\mathcal{I}_q^2}g(x)-1
\end{align*} so $\alpha=1$ and thus \begin{equation*}
    \left(\prod\limits_{g\in\mathcal{I}_q^2}g(x)\right)-1=(x^q-x)^{q-1}.
\end{equation*} Hence $(x-1)\in \mathcal{NC}^{Q_{b(q)}}$ and $\delta_{Q_{B(q)}}(x-1)=\mathcal{P}_{B(q)}$, so $\mathcal{NC}^{Q_{b(q)}}=\{x-1\}$ by Lemma \ref{nonconformlem}.
With this we can characterize all $B(q)$-invariant polynomials as follows: \begin{corollary}\label{bqchar}
For every monic $B(q)$-invariant polynomial $f\in \NR_K^{B(q)}$ exists a unique monic polynomial $F\in K[x]$ and $m\in\mathbb{N}$ with $0\le m<q-1$ such that \begin{equation*}
    f(x)=(x^q-x)^m \cdot F\left(\prod_{g\in\mathcal{I}_q^2}g(x)\right).
\end{equation*} 
\end{corollary}
\begin{remark}
The polynomial \begin{equation*}
    Q(x)=(x^q-x)^{q-1}
\end{equation*} is another quotient map for $B(q)$. For $Q$ the sets $P_{B(q)}$ and $\mathcal{P}_{B(q)}$ remain the same (notice that these sets are always the same regardless which quotient map we choose), just $\mathcal{NC}^Q=\{x\}$ is different. Therefore we can reformulate the previous Corollary in the following way: 
\begin{quote}
       For every monic $B(q)$-invariant polynomial $f\in \NR_K^{B(q)}$ exists a unique monic polynomial $F\in K[x]$ and $m\in\mathbb{N}$ with $0\le m<q-1$ such that \begin{equation*}
    f(x)=(x^q-x)^m \cdot F\left((x^q-x)^{q-1}\right)
\end{equation*}\end{quote} Changing the quotient maps for $B(q)$ in the representation of $B(q)$-invariant polynomials is like changing the basis of a vector space. The polynomials $F$ are, in this analogy, like the coefficients of the vectors written as the linear combination of the basis elements.
\end{remark}
The factorization over finite $K$ can be explained with Theorem \ref{facfin} again: \begin{corollary}
Let $K=\F_{q^s}$ and $F\in\mathcal{I}_{q^s}$ for an $s\in\mathbb{N}\setminus\{0\}$ and $q=p^n$. If $F\neq x-1$ then $F^{Q_{B(q)}}$ has at least \begin{enumerate}
    \item $q$ irreducible factors if $q$ is not prime, i.e. $n>1$, or \item $q-1$ irreducible factors if $q$ is prime, i.e. $n=1$
\end{enumerate} Every such factor has a cyclic stabilizer in $B(q)$ and thus has degree at most $(q-1)\cdot\deg(F)$ if $q$ is not prime and $q\cdot \deg(F)$ if $q$ is prime.
\end{corollary}
\subsection{Projective General Linear Groups} Let $\operatorname{char}(K)=p>0$ and assume that $\F_q\subseteq K$ for $q$ a power of $p$, then $G:=\PG(\F_q)\le \PG(K)$.  First of all we need to calculate a quotient map for $G$. As shown in \cite[Example 3.12]{bluher1} \begin{equation}\label{pgqqg}
    Q_G(x)=\frac{\prod_{r\in\mathcal{I}_q^3}r(x)}{(\prod_{h\in\mathcal{I}_q^1}h(x))^{(q-1)q}} =\frac{\prod_{r\in\mathcal{I}_q^3}r(x)}{(x^q-x)^{(q-1)q}}
\end{equation} is a quotient map for $G$ over every field that contains $\F_q$, so also over $K$. The sets $\mathcal{I}_q^1$ and $\mathcal{I}_q^3$ are the sets of monic irreducible polynomials in $\F_q[x]$ of degree 1 and 3 respectively.  We want to determine $P_G,\mathcal{P}_G$ and $\mathcal{NC}^{Q_G}$. By Lemma \ref{deg2} we know that $P_G\subseteq \F_{q^2}\cup\{\infty\}$ since the equation \begin{equation*}
    [A]\circ v=v
\end{equation*} for $[A]\in\PG(q)$ is, in essence, a polynomial equation over $\F_q$, hence all solutions are algebraic over $\F_q$ (except $\infty$) and thus $[\F_q(v):\F_q]\le 2$. Indeed, $P_G=\F_{q^2}\cup\{\infty\}$ since for $a\in \F_q$ and $v\in\F_{q^2}$ we have \begin{align*}
    G\circ a& =\F_q\cup\{\infty\}\\
    G\circ v&=\F_{q^2}\setminus\F_q.
\end{align*} Therefore $\NR_K^G$ consists of monic polynomials in $K[x]$ with no roots in $\F_q$. The set $\mathcal{P}_G$ contains minimal polynomials of elements in $P_G\setminus (G\circ \infty)=\F_{q^2}\setminus\F_q$, thus we have two cases: \begin{enumerate}
    \item If $\F_{q^2}\not\subseteq K$, then $\mathcal{P}_G=\mathcal{I}_q^2$ and every $g\in \mathcal{I}_q^2$ is also irreducible over $K[x]$
    \item If $\F_{q^2}\subseteq K$, then \begin{equation*}
        \mathcal{P}_G=\{x-v|v\in\F_{q^2}\setminus\F_q\}.
    \end{equation*}
\end{enumerate} Since $B(q)\subseteq G$ and $B(q)$ acts transitively on $\mathcal{P}_G$ so does $G$ and thus $G\ast g= \mathcal{P}_G$ for all $g\in \mathcal{P}_G$. Since $\delta_{Q_G}(\mathcal{NC}^{Q_G})\subseteq \mathcal{P}_G$ we are looking for an irreducible polynomial $F\in\mathcal{I}_K$ such that $$F^{Q_G}=(\prod G\ast h)^k=(\prod \mathcal{P}_G)^k=(\prod\limits_{g\in\mathcal
{I}_q^2}g)^k$$ for an $h\in\mathcal{P}_G$ and $k>1$.  Looking back at Example 3.12 in \cite{bluher1} gives $F=x+1$ and $k=q+1$, thus $\mathcal{NC}^{Q_G}=\{x+1\}$. With Theorem \ref{fac2} we obtain \begin{corollary}\label{pgqchar}
For every monic $\PG(q)$-invariant polynomial $f\in \NR_K^G$ exists a unique monic polynomial $F\in K[x]$ and $m\in\mathbb{N}$ with $0\le m<q+1$ such that \begin{equation*}
    f(x)=\left(\prod\limits_{g\in\mathcal{I}_q^2}g(x)\right)^m \cdot \left((x^q-x)^{(q-1)q\cdot\deg(F)}F\left(\frac{\prod_{r\in\mathcal{I}_q^3}r(x)}{(x^q-x)^{(q-1)q}}\right)\right).
\end{equation*} 
\end{corollary}
For the factorization over finite fields we shortly recall the 3 types of conjugacy classes of cyclic subgroups of $\PG(\F_q)$ (for reference see \cite[Proposition 11.1]{bluher1} or \cite[§8]{endlGruppen}).

Every $[A]\in\PG(q)$ is contained in one of the following three types of conjugacy classes: \begin{enumerate}
    \item $[A]$ fixes a unique element in $\F_q\cup\{\infty\}$, i.e. $[A]\circ v=v$ for a unique $v\in \F_q\cup\{\infty\}$
    \item $[A]$ fixes two different elements in $\F_q\cup\{\infty\}$ under M{\"o}bius-transformation
    \item $[A]$ fixes $\lambda,\lambda^q\in\mathbb{F}_{q^2}\setminus\F_q$ under M{\"o}bius-transformation
\end{enumerate}
We then say that $[A]$ is of type $1,2$ or $3$ respectively. If $[A]$ is of type 1 then $\operatorname{ord}([A])=p$ and $p$ is the prime dividing $q$. Every element $[B]$ of type 2 has an order dividing $q-1$ and if $[C]$ is of type 3 then $\operatorname{ord}([C])|q+1$. So $\mu_{\PG(\F_q)}=q+1$ and we obtain \begin{corollary}\label{pgqfin}
Let $K=\F_{q^s}$ for $s>0$ and $F\in\mathcal{I}_{q^s}$ and $G=\PG(\F_q)$. If $F\neq x+1$ then $F^{Q_G}$ has at least $q^2-q$ irreducible factors and every such factor has degree at most $(q+1)\cdot\deg(F)$.
\end{corollary}
\begin{proof}
 Follows immediately from Theorem \ref{facfin} together with $\mu_{G}=q+1$ and $|G|=q^3-q$. 
\end{proof}
\section*{Acknowledgements}
    I want to thank Alev Topuzoğlu and Henning Stichtenoth for their helpful remarks and the advice they have given me. I am especially grateful to Henning Stichenoth helping me with some technicalities of section 2.1. and making me aware of the paper \cite{madan}. 
    
    I am also very grateful for all of the invaluable help my supervisor Gohar Kyureghyan has given me.  Without her, this paper would probably not exist.

 \printbibliography 

@article{bluher1,
title = {Explicit Artin maps into $PGL_2$},
journal = {Expositiones Mathematicae},
volume = {40},
number = {1},
pages = {45-93},
year = {2022},
author = {Antonia W. Bluher},
}

@article{sticht,
title = {Factorization of a class of polynomials over finite fields},
journal = {Finite Fields and Their Applications},
volume = {18},
number = {1},
pages = {108-122},
year = {2012},
%url = {https://www.sciencedirect.com/science/article/pii/S1071579711000608},
author = {Henning Stichtenoth and Alev Topuzoğlu},
}

@phdthesis{reisphd,
  title={Contemporary topics in Finite Fields: Existence, characterization, construction and enumeration problems},
  author={Lucas Reis},
  year={2018},
  school={Federal University of Minas Gerais},
}

@misc{bassaR,
      title={The R-transform as a power map and its generalisations to higher degree}, 
      author={Alp Bassa and Ricardo Menares},
      year={2019},
      eprint={1909.02608},
      archivePrefix={arXiv},
      primaryClass={math.NT},
}

@article{cohensLem, title={On irreducible polynomials of certain types in finite fields}, volume={66}, number={2}, journal={Mathematical Proceedings of the Cambridge Philosophical Society}, publisher={Cambridge University Press}, author={Cohen, Stephen D.}, year={1969}, pages={335–344},}

@article{cohenConst,
author = {Cohen, Stephen D.},
year = {1992},
month = {06},
pages = {169-174},
title = {The Explicit Construction of Irreducible Polynomials over Finite Fields.},
volume = {2},
journal = {Designs, Codes and Cryptography},
}

@article{meyn,
author={Meyn, Helmut},
title={On the Construction of Irreducible Self-Reciprocal Polynomials Over Finite Fields},
journal={AAECC 1},
year={1990},
pages={43–53},
}

@article{ReisFQG,
author = {Reis, Lucas},
year = {2019},
month = {05},
pages = {169-180},
title = {Möbius-like maps on irreducible polynomials and rational transformations},
volume = {224},
journal = {Journal of Pure and Applied Algebra},
}

@article{GarefalakisGL,
author = {Garefalakis, Theodoulos},
year = {2011},
month = {08},
pages = {1835–1843},
title = {On the action of $GL_2(\mathbb{F}_q)$ on irreducible polynomials over $\mathbb{F}_q$},
volume = {215},
journal = {Journal of Pure and Applied Algebra},
}

@article{Sidel,
author = {Sidel'nikov, Vladimir M.},
year = {1988},
pages = {485–494},
title = {On normal bases of a finite field},
volume = {61 (2)},
journal = {MATH USSR SB},
}

@book{endlGruppen,
  title={Endliche Gruppen 1},
  author={Bertram Huppert},
  series={Grundlehren der mathematischen Wissenschaft},
  year={1967},
  volume={134},
  publisher={Springer-Verlag Berlin Heidelberg},
}

@book{RomanS,
  title={Field Theory},
  author={Steven Roman},
  series={Graduate Texts in Mathematics},
  year={2006},
  publisher={Springer-Verlag New York},
  edition={2},
}

@article{pgl,
 author = {Gregory P. Dresden},
 journal = {Mathematics Magazine},
 number = {3},
 pages = {211--218},
 publisher = {Mathematical Association of America},
 title = {There Are Only Nine Finite Groups of Fractional Linear Transformations with Integer Coefficients},
 volume = {77},
 year = {2004},
}

@article{ReisGLq,
title = {The action of $GL_2(\mathbb{F}_q)$ on irreducible polynomials over $\mathbb{F}_q$, revisited},
journal = {Journal of Pure and Applied Algebra},
volume = {222},
number = {5},
pages = {1087-1094},
year = {2018},
author = {Lucas Reis},
}

@article{abkyu,
title = {Recursive constructions of irreducible polynomials over finite fields},
journal = {Finite Fields and Their Applications},
volume = {18},
number = {4},
pages = {738-745},
year = {2012},
author = {Sergey Abrahamyan and Mahmood Alizadeh and Melsik K. Kyureghyan},
}

@misc{Gow,
  
  url = {https://arxiv.org/abs/2109.06693},
  
  author = {Gow, Rod and McGuire, Gary},
  
  
  title = {On the realization of subgroups of $PGL(2,F)$, and their automorphism groups, as Galois groups over function fields},
  
  publisher = {arXiv},
  
  year = {2021},
  
  copyright = {arXiv.org perpetual, non-exclusive license},
}

@misc{subpglk,
      title={Finite p-Irregular Subgroups of PGL(2,k)}, 
      author={Xander Faber},
      year={2021},
      eprint={1112.1999},
      archivePrefix={arXiv},
      primaryClass={math.NT}
}

@book{permutg,
  title={Permutation Groups},
  author={John D. Dixon and Brian Mortimer},
  series={Graduate Texts in Mathematics},
  year={1996},
  publisher={Springer New York},
  edition={1},
}

@article{comppoly,
author = {Cao, Xiwang and Hu, Lei},
year = {2012},
month = {09},
pages = {},
title = {On the reducibility of some composite polynomials over finite fields},
volume = {64},
journal = {Designs, Codes and Cryptography},
}

@article{KK,
  author = {Melsik K. Kyuregyan and Gohar Kyureghyan},
  journal = {Designs, Codes and Cryptography},
  volume = {61},
  title = {Irreducible Compositions of Polynomials over Finite Fields},
  pages = {301–314},
  publisher = {arXiv},
  year = {2011},
  }

@book{albert,
  title={Fundamental Concepts of Higher Algebra},
  author={Albert, A. Adrian},
  year={1956},
  publisher={University of Chicago Press},
}

@article{Daykin,
 author = {David E. Daykin},
 journal = {The American Mathematical Monthly},
 number = {6},
 pages = {646--648},
 publisher = {Mathematical Association of America},
 title = {Generation of Irreducible Polynomials Over a Finite Field},
 volume = {72},
 year = {1965},
}

@misc{Maurin,
      title={Constructing irreducible polynomials recursively with a reverse composition method}, 
      author={Anna-Maurin Graner and Gohar M. Kyureghyan},
      year={2023},
      eprint={2301.09373},
      archivePrefix={arXiv},
      primaryClass={math.NT}
}

@article{guire,
title = {Invariant rational functions, linear fractional transformations and irreducible polynomials over finite fields},
journal = {Finite Fields and Their Applications},
volume = {79},
pages = {101991},
year = {2022},
author = {Rod Gow and Gary McGuire},

}

@article{reisConstr,
title = {Construction of irreducible polynomials through rational transformations},
journal = {Journal of Pure and Applied Algebra},
volume = {224},
number = {5},
pages = {106241},
year = {2020},
author = {Daniel Panario and Lucas Reis and Qiang Wang},
}

@article{reisEx,
title = {On the existence and number of invariant polynomials},
journal = {Finite Fields and Their Applications},
volume = {61},
pages = {101605},
year = {2020},
author = {Lucas Reis},
}

@article{madan,
author = {Robert C. Valentini and Manohar L. Madan},
journal = {Journal für die reine und angewandte Mathematik},
pages = {156-177},
title = {A Hauptsatz of L. E. Dickson and Artin-Schreier extensions.},
volume = {318},
year = {1980},
}
\end{document}